\documentclass[graybox]{svmult}

\usepackage{mathptmx}       
\usepackage{helvet}         
\usepackage{courier}        
\usepackage{type1cm}        
\usepackage{makeidx}         
\usepackage{graphicx}        
\usepackage{multicol}        
\usepackage[bottom]{footmisc}

\makeindex             

\usepackage{amsmath, amssymb} 
\usepackage{tikz, tikz-cd} 

\usepackage{hyperref} 

\newcommand{\TLdiagscaleLangloisRemillard}{0.25} 

\newcommand{\opdiagscaleLangloisRemillard}{0.75} 
\newcommand{\tikzdessin}[1]{
\begin{tikzpicture}[baseline={(current bounding box.center)},scale = \opdiagscaleLangloisRemillard] 
    #1
\end{tikzpicture}
}

\begin{document}

\title*{Deforming algebras with anti-involution via twisted associativity}
\author{Alexis Langlois-R\'{e}millard}
\institute{Alexis Langlois-R\'{e}millard \at Department of Applied Mathematics, Computer Science and Statistics, Faculty of Sciences, Ghent University, Krijgslaan 281-S9, 9000 Gent, Belgium, \email{alexis.langloisremillard@ugent.be}}

%
%
\maketitle

\abstract{This contribution introduces a framework to study a deformation of algebras with anti-involution. Starting with the observation that twisting the multiplication of such an algebra by its anti-involution generates a Hom-associative algebra of type II, it formulates the adequate modules theory over these algebras, and shows that there is a faithful functor from the category of finite-dimensional left modules of algebras with involution to finite-dimensional right modules of Hom-associative algebras of type II.}

\section{Introduction}
\label{sec:introLangloisRemillard}

An associative algebra $A$ over a commutative, associative and unital ring $R$ is called an algebra with anti-involution if it admits an anti-automorphism $\iota$ that is its own inverse; it is then called an $\iota$-algebra. One of the most common example of such algebras is the algebra of complex square matrices with the conjugate-transpose as anti-involution; for analysts, $C^*$-algebras constitute the backbone of many investigations in functional analysis and operator theory~\cite{dixmier1964LangloisRemillard}. We will be interested in finite-dimensional algebras with an anti-involution and consider them purely from the algebraic point of view. The presence of an anti-involution, not a guaranteed fact if the algebra is not commutative, leads to many interesting properties. For example, ideals of identities of a finitely generated algebra with involution coincide with ideals of identities for a finite dimensional algebra~\cite{sviridova_finitely_2013LangloisRemillard}, and if an associative algebra with involution satisfies some polynomial identities, then the Grassmann envelope of the related superalgebra with superinvolution satisfies the same polynomial properties~\cite{aljadeff_polynomial_2017LangloisRemillard,rizzo_algebras_2020LangloisRemillard}.

The approach considered in this note concentrates on twisting the multiplication of the algebra and emphasize the r\^ole of the anti-involution. It goes along the spirit of Hom-structures, an interesting program of deformations of algebraic structures sparked by the introduction of Hom-Lie algebras by Hartwig, Larsson and Silvestrov~\cite{hartwig_deformations_2004LangloisRemillard}. In the following fifteen years, the notion has been extended to many others algebraic objects: Hom-associative algebras~\cite{makhlouf_hom-algebra_2006LangloisRemillard}, Hom-Poisson algebras~\cite{MakSil2010LangloisRemillard}, Hom-Novikov algebras~\cite{YauNovikovLangloisRemillard}, Hom-Hopf algebras~\cite{makhlouf_hom-algebras_2010LangloisRemillard}, Hom-Weyl algebras~\cite{back_hom-associative_2019LangloisRemillard}, Hom-quantum groups~\cite{yau_hom-quantum_2012LangloisRemillard,yau_hom-quantum_2009-2LangloisRemillard,yau_hom-quantum_2009-3LangloisRemillard}, etc. Many of the Hom-structures keep the properties their classical counterparts have; for example Hom-associative algebras are the universal enveloping algebras of the corresponding Hom-Lie algebras~\cite{yau_enveloping_2008LangloisRemillard}. General programs for investigating Hom-structures via higher algebraic tools have also been proposed via PROPs~\cite{yau_hom-algebras_2011LangloisRemillard}, or via universal algebras and operads~\cite{hellstrom_universal_2014LangloisRemillard}.

We will be interested in (a type of) Hom-associative algebras: an $R$-module with an $R$-linear map $\alpha$ that has a binary operation for which associativity is ``twisted'' by the map $\alpha$. A simple construction by Yau~\cite{yauhomLangloisRemillard} allows to deform any associative algebra with an endomorphism into a Hom-associative algebra. As in an $\iota$-algebra, $\iota$ is an anti-endomorphism, this construction does not give a ``classical'' Hom-associative algebra but instead what Fr\'egier and Gohr call a Hom-associative algebra of type II in their hierarchy~\cite{fregier_hom-type_2010LangloisRemillard}.

The main results and highlights of this contribution are reviewed here. We give a functor from the category of $\iota$-algebras to the category of Hom-associative algebras of type II (proposition~\ref{prop:YautwisttypeIILangloisRemillard}), construct a theory of (finite-dimensional) Hom-modules of Hom-associative algebras of type II, effectively proving that their category is abelian (proposition~\ref{prop:hommoduleII1LangloisRemillard}--\ref{prop:hommoduleII5LangloisRemillard}); and give a faithful functor between the category of finite-dimensional left-modules of a $\iota$-algebra to the category of finite-dimensional right-modules of the associated Hom-associative algebra of type II (proposition~\ref{prop:equivalenceofmodulescategoriesLangloisRemillard}). 

The contribution is organized as follow. Section~\ref{sec:alginvolutionLangloisRemillard} presents basic definitions for $\iota$-algebras and then quickly surveys cellular algebras. Section~\ref{sec:homstruLangloisRemillard} first introduces the vocabulary of Hom-structures for comparison purposes and then proceed to construct a Hom-modules theory for Hom-associative algebras of type II and to give an example on a family of diagrammatic $\iota$-algebras: the Temperley-Lieb algebras~\cite{TLLangloisRemillard} viewed by their cellular basis~\cite{GraCel96LangloisRemillard}. Finally, section~\ref{sec:discussionLangloisRemillard} informally discusses an idea for the study of this type of structures in more generality via diagrammatic formalism, and gives an example of application by illustrating alternate proofs of some equivalences from Fr\'egier's and Gohr's hierarchy~\cite{fregier_hom-type_2010LangloisRemillard}.

\section{Algebras with involution}
\label{sec:alginvolutionLangloisRemillard}

In this section we present some vocabularies regarding $\iota$-algebra and briefly introduce cellular algebras. For the remaining of the note, let $R$ be an associative, unital and commutative ring. 
\begin{definition}
	Let $A$ be an associative and unital $R$-algebra. Let $\iota: A\to A$ be an anti-morphism of algebras such that $\iota^2=1$. The pair $(A,\iota)$ is called an \emph{algebra with anti-involution}\footnote{It is common to call them algebra with involution~\cite{tignol_algebras_1998LangloisRemillard,dixmier1964LangloisRemillard}. We added the prefix ``anti'' to avoid the confusion and emphasize the fact that the map $\iota$ is an anti-automorphism.}	or an $\iota$-algebra.
\end{definition}

A short note is in order here. Usually the involution is a conjugate-linear map ($\iota(r\cdot a) = \overline{r}\cdot\iota(a)$) to mimic the conjugate-transpose operator of complexes matrices. Of course, this only makes sense if $R$ has a conjugation, for example if one works on $\mathbb{C}$, as for $C^*$ algebras. For $R$-algebras, the custom is to only ask for a $R$-linear map (for example as in~\cite{sviridova_finitely_2013LangloisRemillard}). 

To make it explicit, the anti-involution $\iota$ now satisfies, for any elements $a,b\in A$ and $r,s\in R$,
\begin{align}
	\iota(ab) &= \iota(b)\iota(a); & \iota(\iota(a)) &= a; & \iota(r\cdot a + s\cdot b) = r\cdot\iota(a) + s\cdot \iota(b).
\end{align}

Morphisms of $\iota$-algebras are algebras morphisms commuting with the two anti-involutions.

\begin{definition}
	For two $\iota$-algebras $(A_1,\iota_1)$ and $(A_2,\iota_2)$, a morphism of algebras $\phi: A_1\to A_2$ is called a \emph{morphism of $\iota$-algebras} if it commutes with the anti-involutions:
	\begin{equation}\label{eq:morphisminvolutivealgebrasLangloisRemillard}
		\phi\circ \iota_1 = \iota_2\circ \phi.
	\end{equation}
\end{definition}

Ideals of $\iota$-algebras have the additional restriction of being fixed by the anti-involution.

\begin{definition}
	An algebraic ideal $J$ of $A$ is an \emph{$\iota$-ideal} of the $\iota$-algebra $(A,\iota)$ if $\iota(j) \in J$ for all $j\in J$.
\end{definition}

For an $\iota$-ideal $J$, the quotient algebra $A/J$ is also an $\iota$-algebra. Note that any $\iota$-algebra $A$ decomposes into symmetric ($\iota(s) = s$) and skew-symmetric ($\iota(t) = -t$) parts by taking $A^+$ to be the set of all elements $a+\iota(a)$, $a\in A$ and $A^-$, the set of all elements $a-\iota(a)$, $a\in A$. Furthermore, $A^+$ is a Jordan algebra when endowed with the anti-commutator $\{a,b\} = ab+ba$, and $A^-$ is a Lie algebra with the commutator $[ a,b] = ab-ba$.

We close the section by presenting a motivating example that will be used in the end of section~\ref{sec:homstruLangloisRemillard}: cellular algebras. The notion of cellular algebras was introduced by Graham and Lehrer in their seminal paper~\cite{GraCel96LangloisRemillard}. In a rough statement, a cellular algebra is an associative finite-dimensional unital algebra that admits an anti-involution and a basis stratifying the algebra according to a certain poset while behaving correctly under the anti-involution. Many important families of algebras admit an interesting cellular basis: Temperley-Lieb algebras, Brauer algebras, most Hecke algebras, etc. Once a cellular basis is exhibited for an algebra, the hard problems of exhibiting a complete family of simple modules and giving the composition multiplicities of the indecomposable projective reduce to linear algebra problems. We also present a basis-free definition based on a ring-theoretical framework given by K\"onig and Xi~\cite{KonigXiCellLangloisRemillard} that emphasizes the importance of the anti-involution. A good review on the subject is the second chapter of Mathas's book on Iwahori-Hecke algebras~\cite{Mathas99LangloisRemillard}.

\begin{definition}[Graham and Lehrer~\cite{GraCel96LangloisRemillard}]\label{dfn:cellularGLLangloisRemillard}
 Let $R$ be a commutative associative unital ring. An associative $R$-algebra $A$ is called {\em cellular} if it admits a cellular datum $(\Lambda,M,C,\iota)$ consisting of the following:
 \begin{enumerate}
  \item a partially-ordered set $\Lambda$ and, for each $d\in\Lambda$, a finite set $M(d)$;
  \item an injective map $C:\bigsqcup_{d\in\Lambda} M(d)\times M(d) \to A$ whose image is an $R$-basis of $A$;
  \item an anti-involution $\iota: A \to A$ such that
  	\begin{equation}\label{eq:axiomcellularityinvolutionLangloisRemillard}
    \iota(C(s,t)) = C(t,s), \qquad \text{for all } s,t\in M(d);
  \end{equation}
  \item if $d\in\Lambda$ and $s,t\in M(d)$, then for any $a\in A$,
 \begin{equation}\label{eq:axiomcellularityLangloisRemillard}
   aC(s,t) \equiv \sum_{s'\in M(d)} r_a(s',s) C(s',t) \mod A^{>d},
 \end{equation}
where $A^{>d} = \left\langle C^{e}(p,q) \mid e>d \mid p,q\in M(e)\right\rangle_R$ and  $r_a(s',s)\in R$.
 \end{enumerate}
\end{definition}
\noindent The anti-involution $\iota$, together with \eqref{eq:axiomcellularityLangloisRemillard}, yields the equation:

\begin{equation}\label{eq:staronaxiomcellularityLangloisRemillard}
  C(t,s)a^* \equiv \sum_{s'\in M(d)} r_a(s',s) C(t,s') \mod A^{>d},
\end{equation}
for all $s,t\in M(d)$  and  $a\in A$.


To be completely transparent, there are two small differences that became customary (for example~\cite{dengfiniteLangloisRemillard, Mathas99LangloisRemillard}) between this definition and the one of Graham and Lehrer: first the partial order is reversed, as to better compare with quasi-hereditary algebras~\cite{cline1988finiteLangloisRemillard} and second, the poset $\Lambda$ is not finite (but all of the sets $M(d)$ are finite). Notable generalizations and deformations of the notion of cellularity include: relative cellularity, where multiple partial orders grade the algebra~\cite{ehrig_relative_2019LangloisRemillard}; affine cellularity, where the notion is extended to infinite dimensional algebras~\cite{koenig_affine_2012LangloisRemillard,GL-AffLangloisRemillard}, and almost cellular algebra, where the anti-involution is replaced by a special filtration of the algebra~\cite{guay_almost_2015LangloisRemillard}.

As a simple example, consider the polynomial algebra $\mathbb C[x]$. Let $\Lambda = \mathbb N = \{0,1,\dots\}$, for $n\in \mathbb N$ let $M(n):= \{ n\}$ and define $C:M\times M \to \mathbb C[x]$ by $n\times n \mapsto x^n$. The algebra is commutative so take the (anti-)involution simply to be the identity. The image of $C$ is obviously a basis of $\mathbb{C}[x]$. Axiom~\eqref{eq:axiomcellularityinvolutionLangloisRemillard} is trivially satisfied because all the sets $M(n)$ are singletons. Axiom~\eqref{eq:axiomcellularityLangloisRemillard} simply states that multiplying two non-trivial polynomials will yield a polynomial of higher degree. It is one of the simplest example of cellular structure.

The previous example downplays the importance of the anti-involution in the structure of cellular algebras. The equivalent basis-free definition of K\"onig and Xi highlights its key importance.
\begin{definition}[K\"onig and Xi~\cite{KonigXiCellLangloisRemillard}]\label{dfn:cellularKXLangloisRemillard}
	Let $A$ be an $R$-algebra where $R$ is a commutative Noetherian integral domain. Assume there is an anti-involution $\iota$ on $A$. A two-sided ideal $J$ of $A$ is called a {\em cell ideal} if:
	\begin{enumerate}
		\item it is an $\iota$-ideal ($\iota(J)=J$);
		\item there exists a left ideal $\Delta \subset J$ such that $\Delta$ is finitely generated and free over $R$;
		\item there is an isomorphism of $A$-bimodules $\psi: J \xrightarrow{\sim} \Delta \otimes_R \iota(\Delta)$ making the following diagram commutative:
		\begin{equation}
		\begin{tikzcd}[column sep = large, row sep = large]
			J \arrow[d,"\iota"] \arrow[r, "\psi"] & \Delta \otimes_R
			\iota(\Delta) \arrow[d,"x\otimes y \mapsto \iota(y)\otimes \iota(x)"]\\
			J \arrow[r,"\psi"] & \Delta \otimes_R \iota(\Delta)
		\end{tikzcd}.
		\end{equation}
	\end{enumerate}

The algebra $A$ together with the anti-involution $\iota$ is called {\em cellular} if there is an $R$-module decomposition $A$ = $J_1' \oplus J_2' \oplus \dots \oplus J_n'$ with $\iota(J_j')=J_j'$ for each $j$ and such that setting $J_j = \bigoplus_{k=1}^j J_k'$ gives a chain of two-sided ideals of $A$
\begin{equation}
	0 = J_0 \subset J_1 \subset J_2 \subset \dots \subset J_n = A
\end{equation}
in which each quotient $J_j' = J_j/J_{j-1}$ is a cell ideal of the quotient $A/J_{j-1}$ with respect to the restriction of $\iota$ on the quotient. 
\end{definition}

\begin{proposition}[K\"onig and Xi~\cite{KonigXiCellLangloisRemillard}]
	The two definitions of cellular algebra are equivalent.
\end{proposition}

From the cellular basis it is  possible to construct a family of \emph{cell modules}\footnote{Also often called \emph{standard modules} to emphasize the links with Hecke algebras and quasi-hereditary algebras.}. Each cell module $\mathsf{C}_{d}$ admits a symmetric and invariant bilinear form $\phi^d(-,-)$.

Define, for any cellular algebra, $\Lambda^0$ to be the subset of $\Lambda$ in which the bilinear form just defined is not identically zero. The \emph{radical} of the bilinear form $\phi^d$ is denoted $\mathsf{R}^{d}$. As the form is invariant, it is a submodule of $\mathsf{C}_{d}$. However, there is even more to it.

\begin{proposition}[Graham and Lehrer, Prop.~3.2 and Thm 3.4~\cite{GraCel96LangloisRemillard}]\label{prop:completesetofsimpleLangloisRemillard}
	Let $A$ be a cellular algebra over a field and $d\in \Lambda^0$. The radical $\mathsf{R}^{d}$ of the bilinear form $\phi^d$ is the Jacobson radical of $\mathsf{C}_{d}$; the quotient $\mathsf{I}^{d} := \mathsf{C}_{d}/\mathsf{R}^{d}$ is absolutely irreducible and $\{\mathsf{I}^{d} \mid d\in \Lambda^0\}$ is a complete set of non-isomorphic (absolutely) irreducible modules.
\end{proposition}

The main theorem links the decomposition factors of the indecomposable projective modules and the irreducible ones. Recall that $[M:I]$ is the composition multiplicity of the simple module $I$ in the module $M$, that is the number of simple quotients isomorphic to $I$ in the composition series of $M$. Denote by $\mathsf{P}^{d}$ the projective cover of $\mathsf{I}^{d}$ and let $\mathbf{D} = ([\mathsf{C}_{d}:\mathsf{I}^{e}])_{d\in \Lambda,e\in \Lambda^0}$ be the decomposition matrix of $A$ and $\mathbf C = ([\mathsf{P}^{d}:\mathsf{I}^{e}])_{d,e\in\Lambda^0}$, its Cartan matrix.

\begin{theorem}[Graham and Lehrer, Thm. 3.7~\cite{GraCel96LangloisRemillard}]\label{thm:cisddLangloisRemillard}
	The matrices $\mathbf C$ and $\mathbf D$ are related by $\mathbf C = \mathbf D^t \mathbf D$.
\end{theorem}

Therefore, one can characterize the composition series of the indecomposable projective modules of a cellular algebra by simpler linear algebraic tools. A simple criterion for semisimplicity follows: should the radical of the bilinear form of each cell module be trivial, then the algebra is semisimple.


\section{Hom-structures}
\label{sec:homstruLangloisRemillard}

The notion of Hom-associative algebras comes from Makhlouf and Silvestrov~\cite{makhlouf_hom-algebra_2006LangloisRemillard} and the ten complete cases of possible Hom-associativity appearing from possible choices of parentheses and twisting map placement on the associativity equation $(ab)c=a(bc)$ were given later by Fr\'egier and Gohr~\cite{fregier_hom-type_2010LangloisRemillard}. 

Two types of Hom-associativity are considered here. The first one is the ``classical'' Hom-associativity that are the parallel of universal enveloping algebras for Hom-Lie algebras. It is characterized by equation~\eqref{eq:homassILangloisRemillard}. In the language of Fr\'egier and Gohr, it is called of type I\textsubscript{1}. The content presented in this section is mostly standard and can be found in many references, see for example~\cite{makhlouf_hom-algebra_2006LangloisRemillard, back_hom-associative_2019LangloisRemillard,makhlouf_hom-algebras_2010LangloisRemillard}. Only the part on representation theory is slightly less common and we defer to B\"ack and Richter~\cite{back_non-associative_2018LangloisRemillard} for a careful overview of module theory in general. 

The second section is devoted to Hom-associativity of type II and some constructions around them. The associativity deformation is given by equation~\eqref{eq:homassIILangloisRemillard}. This subject was mainly studied (e.g. in~\cite{fregier_unital_2009LangloisRemillard}) in the context of unital multiplicative Hom-associative algebras. The setup employed here will be of weakly unital Hom-associative algebras with the twisting map an anti-involution, thus enabling a functorial construction (proposition~\ref{prop:YautwisttypeIILangloisRemillard}) akin to Yau's twisting principle (proposition~\ref{prop:YautwistingLangloisRemillard}) for $\iota$-algebras that also extends to a functor on modules (proposition~\ref{prop:equivalenceofmodulescategoriesLangloisRemillard}).

The last section studies a concrete example: Temperley-Lieb algebras. Employing its diagrammatic formulation give a very natural example of the module theory employed here and we shall see that some of the structures coming from its cellularity are preserved by the functor in the semisimple case at least.

\subsection{Review of Hom-associative results}
\label{subsec:reviewhomLangloisRemillard}

\begin{definition}\label{def:homassLangloisRemillard}
	A {\em Hom-associative algebra} over an associative, commutative and unital ring $R$ is a triple $(A,\cdot,\alpha)$ consisting of an $R$-module $A$, an $R$-bilinear binary operation $\cdot: A\times A \to A$ and an $R$-linear map $\alpha: A \to A$ satisfying 
	\begin{equation}\label{eq:homassILangloisRemillard}
		\alpha(a) \cdot (b\cdot c) = (a\cdot b) \cdot \alpha(c)
	\end{equation}
	for all $a$, $b$, $c \in A$.
\end{definition}
The map $\alpha$ is referred to as the {\em twisting map}. When it is an homomorphism, the Hom-algebra is said to be {\em multiplicative}. 

A Hom-associative $R$-algebra $A$ is said to be {\em weakly left unital} if there exists $e_{\ell}\in A$ such that $e_{\ell} \cdot a = \alpha(a)$ for all $a\in A$; it is said to be {\em weakly right unital} if there exists $e_r\in A$ such that $a\cdot e_r = \alpha(a)$ for all $a\in A$, and it is deemed {\em weakly unital} if there exists $e\in A$ that is both a weak left unit and a weak right unit. Beware, the word \emph{unital} is reserved for an algebra with a unit $\mathrm{id}$, that is $x\cdot \mathrm{id} = \mathrm{id} \cdot x = x$, for any $x\in A$.

There is a canonical way, known as \textit{Yau's twisting}, of defining a weakly unital Hom-associative algebra from a unital associative algebra.

\begin{proposition}[Yau~\cite{yauhomLangloisRemillard}]\label{prop:YautwistingLangloisRemillard}
Let $A$ be an associative algebra with unit $1_A$ and $\alpha: A \to A$ be a endomorphism. Defining the operation $\star: A\times A \to A$ by $a\star b = \alpha(a\cdot b)$ gives a Hom-associative algebra $(A,\star, \alpha)$ with weak unit $1_A$. 
\end{proposition}

\begin{definition}
The notion of Hom-algebras morphism between two Hom-associative algebras $(A_1,\cdot_1,\alpha_1)$ and $(A_2,\cdot_2,\alpha_2)$ is given by an $R$-modules morphism $f:A_1\to A_2$ satisfying
\begin{equation}
	f\circ \alpha_1(x) = \alpha_2 \circ f(x); \quad f(x)\cdot_2 f(y) = f(x\cdot_1 y)
\end{equation}
for all $x,y\in A_1$.	
\end{definition}

With this definition, Yau's twisting extends to a functor from the category of associative algebras to the one of Hom-associative algebras~\cite{yauhomLangloisRemillard}.

Ideals of Hom-associative algebras must behave well under the twisting map.
\begin{definition}
A \emph{Hom-ideal} is an algebraic ideal fixed by $\alpha$. So for any element $a\in A$ and $j\in J$, the multiplication $a\cdot j\in J$ and $\alpha(j) \in J$.

We call a Hom-algebra \emph{Hom-simple} if it has no non-trivial Hom-ideal. 
\end{definition}

If the algebra is simple, then it is Hom-simple, but the converse is not true as Hom-ideal are a stronger notion, which in turn makes the notion of Hom-simplicity weaker.

Now, we turn to some vocabulary of module theory for comparison purpose with the type II case (see~\cite{back_non-associative_2018LangloisRemillard} for more details).
\begin{definition}\label{def:hommoduleILangloisRemillard}
	Let $(A,\cdot,\alpha)$ be a Hom-associative $R$-algebra. A triple $(V,\cdot_V,\alpha_V)$ is said to be a \emph{(left) Hom-module} if $V$ is an $R$-module; the operation $\cdot_V : A \times V \to V$ is $R$-linear, and $\alpha_V: V \to V$ is an $R$-linear map such that 
		\begin{equation}\label{eq:hommoduleILangloisRemillard}
			(a\cdot b)\cdot_V \alpha_V(v) = \alpha(a)\cdot_V (b\cdot_V v),
		\end{equation}
		for all $a,b\in A$ and $v\in V$.
\end{definition}

\begin{definition}
	Let $(V,\alpha_V)$ and $(U,\alpha_U)$ be two Hom-$A$-modules. Let $\phi : U \to V$ be a linear map. It is a \emph{morphism of Hom-$A$-modules} if it also respects 
	\begin{equation}
		\phi(au) = a\phi(u), \quad \alpha_V(\phi(u)) = \phi(\alpha_U(u)).
	\end{equation}

\end{definition}

A submodule $N$ of $M$ is a \emph{Hom-submodule} if it is invariant under the map $\alpha_M$. Many usual properties of modules hold: intersection, union, image and preimage under morphism, quotient, and the first, second and third isomorphism theorems, as shown in~\cite{back_non-associative_2018LangloisRemillard}. 

From this, we can define a \emph{Hom-simple module} as a module with no non-trivial Hom-submodule and a \emph{Hom-semisimple algebra} as a Hom-algebra that decompose into a sum of Hom-simple modules when viewed as a Hom-module over itself. 

In conclusion. remark that the (left) Hom-modules along with their morphisms form an abelian category~\cite{zhang_remarks_2015LangloisRemillard}.

\subsection{Hom-associativity of type II}
 Hom-associativity of type II introduces a slight change in the order of deformations by the twisting map. Hygienic procedures are required to ensure the correct definitions for equivalent objects of the preceding section.

\begin{definition}
A \emph{Hom-associative algebra of type II} over an associative, commutative and unital ring $R$ is a triple $(A,\cdot,\alpha)$ consisting of an $R$-module $A$, an $R$-bilinear binary operation $\cdot: A\times A \to A$ and an $R$-linear map $\alpha: A \to A$ satisfying
\begin{equation}\label{eq:homassIILangloisRemillard}
x \cdot \alpha(y\cdot z) = \alpha(x\cdot y) \cdot z,
\end{equation}
for all $x,y,z\in A$.
\end{definition}

The map $\alpha$ is still referred to as the \emph{twisting map}; the notion of weakly unitality stays the same, and when $\alpha$ is an anti-endomorphism, the Hom-algebra is said to be \emph{anti-multiplicative}.

That Hom-associative algebras of type II are an interesting avenue to deform $\iota$-algebras is illustrated by the following proposition.

\begin{proposition}\label{prop:YautwisttypeIILangloisRemillard}
	Let $A$ be a unital associative $R$-algebra and $\iota: A\to A$ be an anti-involution. Consider the binary operation $\circledast : A\times A \to A$ defined by the mapping $(x,y) \mapsto \iota(x\cdot y) = \iota(y)\cdot \iota(x)$. If $e\in A$ is the unit of $A$, then the triple $(A,\circledast,\iota)$ is an anti-multiplicative Hom-associative algebra of type II with weak unit $e$.
\end{proposition}
\begin{proof}
\smartqed
 First, $A$ is an $R$-module as it is an $R$-algebra and $\iota$ is an $R$-linear map. That equation~\eqref{eq:homassIILangloisRemillard} holds follows from simple algebraic manipulations. Let $x,y,z\in A$.
	\begin{align*}
    	x\circledast\iota(y\circledast z) &= x\circledast \iota(\iota(y\cdot z)) && \text{definition of } \circledast \\
    	&= x\circledast(y\cdot z) && \iota \text{ is involutive}\\
    	&= \iota(x\cdot (y\cdot z)) && \text{definition of }\circledast\\
    	&= \iota((x\cdot y)\cdot z) && \text{associativity in } A\\
    	&= (x\cdot y) \circledast z && \text{definition of }\circledast \\
    	&= \iota^2(x\cdot y)\circledast z &&\iota\text{ is involutive}\\
    	&= \iota(x\circledast y)\circledast z && \text{definition of } \circledast.
    \end{align*}
    It is thus a Hom-associative algebra of type II. It is weakly unital, for $e$ being an unit in $A$ implies 
    \[
    e\circledast x = \iota(e\cdot x) = \iota(x) = \iota(x\cdot e) = x\circledast e.
    \]\qed
\end{proof}

Before continuing, it is worth noting that Yau's construction (proposition~\ref{prop:YautwistingLangloisRemillard}) does not work on anti-multiplicative Hom-associative algebra of type I\textsubscript{1}. Indeed, if one deforms multiplication in an associative algebra $A$ with anti-endomorphism $\alpha$, one gets for $a,b,c\in A$
\begin{align*}
	\alpha(a) \star (b\star c) &= \alpha(a) \star \alpha(b\cdot c)\\
	&= \alpha(\alpha(a)\cdot \alpha(c) \cdot \alpha(b))
	\intertext{on one hand, and}
	(a\star b)\star \alpha(c) &= \alpha(a\cdot b) \star \alpha(c)\\
	&= \alpha(\alpha(b)\cdot \alpha(a) \cdot \alpha(c))
\end{align*}
on the other. When $\alpha$ is an anti-involution for example, this amounts to $b\cdot c \cdot a = c\cdot a \cdot b$, which does not hold generally if $A$ is not commutative.

The notions of morphisms between Hom-associative algebras and of Hom-ideal have a direct equivalent for type II.

\begin{definition}
	Let $(A_1, \cdot_1, \alpha_1)$ and $(A_2,\cdot_2,\alpha_2)$ be two Hom-associative algebras of type II. We call an $R$-linear map $\phi : A_1 \to A_2$ a \emph{Hom-associative algebras morphism} if 
	\begin{equation}\label{eq:morphismII1LangloisRemillard}
		\phi\circ \alpha_1(x) = \alpha_2 \circ \phi(x)
	\end{equation}
 and 
  \begin{equation}\label{eq:morphismII2LangloisRemillard}
  	\phi(x)\cdot_2\phi(y) = \phi(x\cdot_1y).
  \end{equation}
  It is a \emph{Hom-associative algebras anti-morphism} if instead of the last equation, it respects
  \begin{equation}\label{eq:antimorphismII2LangloisRemillard}
  	\phi(x\cdot_1 y) = \phi(y)\cdot_2\phi(x).
  \end{equation}
\end{definition}

\begin{definition}
	An algebraic ideal $J$ of a Hom-associative algebra $(A,\cdot, \alpha)$ is called a  \emph{left Hom-ideal} if it is fixed by $\alpha$. So for any $j\in J$ and $a\in A$ it must be that $a\cdot j \in J$ and $\alpha(j) \in J$. A \emph{right Hom-ideal} is an $\alpha$-invariant right algebraic ideal and a \emph{two-sided Hom-ideal} is both a left and a right Hom-ideal.
\end{definition}

As in the type I\textsubscript{1} case, proposition~\ref{prop:YautwisttypeIILangloisRemillard} extends to a functor $F$ between the $\iota$-algebras and Hom-associative algebras of type II. Indeed for two $\iota$-algebras $(A,\iota_A)$ and $(B,\iota_B)$ a morphism $\phi:A\to B$ of $\iota$-algebras becomes a morphism of Hom-associative algebras under $F$ by making 
\begin{equation}
\begin{aligned}
	F(\phi): (A,\circledast_A,\iota_A) &\longrightarrow (B,\circledast_B,\iota_B)\\
	a &\longmapsto \phi(a),
\end{aligned}
\end{equation}
because then as a morphism of $\iota$-algebras, $\phi$ commutes with the anti-involutions and thus equation~\eqref{eq:morphismII1LangloisRemillard} amounts to equation~\eqref{eq:morphisminvolutivealgebrasLangloisRemillard}. Finally, working out the operations shows that equation~\eqref{eq:morphismII2LangloisRemillard} is respected. Let $x,y\in A$
\begin{align*}
	F(\phi)(x\circledast_A y) &= \phi(x\circledast_A y)\\
	&= \phi(\iota_A(y))\phi(\iota_A(x))\\
	&= \iota_B(\phi(y))\iota_B(\phi(x))\\
	&= \phi(x)\circledast_B \phi(y) = F(\phi)(x)\circledast_B F(\phi)(y).
\end{align*}

One must express cautions while defining modules for type II Hom-algebras. Indeed, the interaction between equations~\eqref{eq:hommoduleILangloisRemillard} and~\eqref{eq:homassIILangloisRemillard} constrains a lot the possible modules: in particular, a Hom-algebra of type II would not be a module on itself if one would use definition~\ref{def:hommoduleILangloisRemillard} because then it would be required that
\begin{equation*}
(a\cdot b)\cdot \alpha(c) = \alpha(a) \cdot (b \cdot c),
\end{equation*}
so type I\textsubscript{1} Hom-associativity, which is not in general a consequence of type II Hom-associativity. As Fr\'egier and Gohr remarked, this would hold if $\alpha$ was an abelian group morphism and the algebra unital~\cite{fregier_hom-type_2010LangloisRemillard} (we show this diagrammatically at the end of section~\ref{sec:discussionLangloisRemillard}). We do not impose such restrictions and the two concepts are different in general.

Therefore it seems more appropriate to use a slightly different definition. The switch to right modules is to stay coherent with proposition~\ref{prop:equivalenceofmodulescategoriesLangloisRemillard}, obviously such results will also hold for left modules.

\begin{definition}
	Let $(A,\cdot,\alpha)$ be a Hom-associative algebra of type II. The triple $(V,\cdot_V,\alpha_V)$ is said to be a (right) \emph{Hom-module} if $V$ is a $R$-module, there is an action $\cdot_V: A \times V \to V$ and an $R$-linear map $\alpha_V : V\to V$ that respect
	\begin{equation}\label{eq:homoduleIILangloisRemillard}
		\alpha_V(v\cdot_V b)\cdot_V a =  v\cdot_V \alpha(a\cdot b).
	\end{equation}
\end{definition}

In this way, any Hom-associative algebra of type II $(A,\cdot,\alpha)$ is also a Hom-module $(A,\cdot,\alpha)$ on itself. The associated concepts of submodule and morphism are defined below.

\begin{definition}
	Consider a Hom-module $(V,\cdot_V,\alpha_V)$ of a Hom-associative algebra of type II $(A,\cdot,\alpha)$. An additive subgroup $U$ of $V$ is called a \emph{Hom-submodule} if it is closed under the scalar multiplication of $V$ and $\alpha_V(U) \subset U$.
\end{definition}

\begin{definition}
	Let $(V,\cdot_V, \alpha_V)$ and $(W,\cdot_W,\alpha_W)$ be two Hom-modules. Then an $R$-linear map $\phi:V\to W$ is called a \emph{morphism of Hom-modules} if it also respects
	\begin{equation}\label{eq:morphismhommoduleIILangloisRemillard}
		\phi(v\cdot_V a) = \phi(v)\cdot_Wa, \quad \alpha_W(\phi(v)) = \phi(\alpha_V(v)).
	\end{equation}
\end{definition}

The following propositions prove that the category of (finite dimensional) Hom-modules is abelian, enabling the general results that go with it: isomorphism theorems, exact sequences, diagram-chasing, etc. There should be no surprise here as Hom-modules are at their core modules over a ring. We will not delve too deeply in these considerations, they are to be seen mostly as a safeguard to prevent abuse. Not a lot of changes appear in the proofs from type I\textsubscript{1} as can be seen by comparing what follows with the survey of Hom-modules theory in B\"ack and Richter~\cite{back_non-associative_2018LangloisRemillard}.

For the following, let $(V,\cdot_V,\alpha_V)$ and $(W,\cdot_W,\alpha_W)$ be two (right) Hom-modules of an Hom-associative algebra of type II $(A,\cdot,\alpha)$. They will be denoted respectively by the slight abuses of notation $V$, $W$, and $A$.

\begin{proposition}\label{prop:hommoduleII1LangloisRemillard}
	Let $\phi: V\to W$ be a morphism of Hom-modules, and let $V'\subset V$ and $W'\subset W$ be respectively a Hom-submodule of $V$ and a Hom-submodule of $W$. The image $\phi(V')$ is a Hom-submodule of $W$ and the preimage $\phi^{-1}(W')$ is a Hom-submodule of $V$.
\end{proposition}
\begin{proof}
	\smartqed
	That $\phi(V')$ and $\phi^{-1}(W')$ are subgroups of their respective space comes from the fact that $\phi$ is an $R$-linear map. Let $a\in A$ and $w\in \phi(V')$. Consider a preimage $v\in V'$ of $w$. Then, $w \cdot_W a = \phi(v)\cdot_W a = \phi(v\cdot_V a)\in \phi(V')$ and $\alpha_W(w) =\alpha_W(\phi(v)) = \phi(\alpha_V(v))\in \phi(V')$ because $\phi$ is a morphism and $V'$ is a Hom-submodule of $V$, and thus fixed by $\alpha_V$. 
	
	For $u\in \phi^{-1}(W')$, there is an element $x\in W'$ such that $\phi(u) =x$. Acting by $a\in A$ on $u$ stays in $\phi^{-1}(W')$ for $\phi(u\cdot_V a) = \phi(u)\cdot_W a = x\cdot_W a \in W'$ as $W'$ is a Hom-submodule of $W$. Furthermore $\alpha_V(u)$ is in $\phi^{-1}(W')$ for $\phi(\alpha_V(u)) =\alpha_W(\phi(u))\subset W'$.
	\qed
\end{proof}

\begin{proposition}\label{prop:hommoduleII2LangloisRemillard}
	Any intersection of Hom-submodules is a Hom-submodule.
\end{proposition}
\begin{proof}
\smartqed
	Let $\{ U_i\}_{i\in I}$ be a set of Hom-submodules of $V$. Note $U = \bigcap_{i\in I} U_i$. If $U = \emptyset$, then it is trivially a Hom-submodule. Assume it is non-empty. It is an additive group. Let $a\in A$ and $u\in U$. Fix $i\in I$. Then $u\cdot a \in U_i$ and $\alpha_V(u) \in U_i$ because $U_i$ is a submodule. As $i\in I$ is arbitrary, then $a\cdot u \in U$ and $\alpha_V(u)\in U$. \qed
\end{proof}

\begin{proposition}\label{prop:hommoduleII3LangloisRemillard}
	A finite sum of Hom-submodules is a Hom-submodule.
\end{proposition}
\begin{proof}
\smartqed
	Let $U_1, \dots, U_k$ be Hom-submodules of $V$. Let $U = \sum_{i=1}^k U_i$. Point-wise sum and $A$-action turn it in a Hom-submodule. Indeed $\alpha_V(\sum_{i=1}^k u_i) = \sum_{i=1}^k \alpha_V(u_i)\in U$.\qed
\end{proof}

\begin{proposition}\label{prop:hommoduleII4LangloisRemillard}
	Consider finitely many Hom-modules $U_1, \dots, U_k$ of $A$. The set $U= \bigoplus_{i=1}^k U_i$ is a Hom-submodule with the action $\cdot: U \times A \to U$ given by $(u_1,\dots, u_k)\cdot a = (u_1 \cdot_{U_1} a, \dots, u_k \cdot_{U_k}a)$ and the action of $\alpha_U : U\to U$ given by $\alpha_{U}(u_1,\dots , u_k) = (\alpha_{U_1}(u_1), \dots, \alpha_{U_1}(u_k))$.
\end{proposition}
\begin{proof}
\smartqed
	The only point that requires proving is the good interaction of $\cdot_U$ and $\alpha_U$. Let $a,b\in A$ and $u = (u_1,\dots, u_k)\in U$. 
	\begin{align*}
		u\cdot_U \alpha(a\cdot b) &= ( u_1 \cdot_{U_1} \alpha(a\cdot b), \dots , u_k \cdot_{U_k} \alpha(a\cdot b))\\
		&= (\alpha_{U_1}(u_1\cdot_{U_1}b )\cdot_{U_1}a, \dots , \alpha_{U_k}(u_k \cdot_{U_k} b)\cdot_{U_k} a)\\
		&=  (\alpha_{U_1}(u_1\cdot_{U_1}b), \dots , \alpha_{U_k}(u_k \cdot_{U_k} b))\cdot_U a\\
		&=  \alpha_U(u\cdot_U b)\cdot_U a.
	\end{align*}
	And thus equation~\eqref{eq:homoduleIILangloisRemillard} holds. The definition indicates clearly that $U$ will be invariant under $\alpha_V$ as each $U_i$ is a Hom-submodule.\qed
\end{proof}

\begin{proposition}\label{prop:hommoduleII5LangloisRemillard}
	Let $U$ be a Hom-submodule of $V$. The quotient $V/U$ is a well-defined Hom-module under the action and the map given by
	\begin{equation}
	\begin{aligned}
	\cdot_{V/U}: V/U \times  A &\to V/U	 & \alpha_{V/U}: V/U &\to V/U\\
	(v+U,a) &\mapsto v \cdot_{V/U}a + U,  & v+U &\mapsto \alpha_V(v) + U.
	\end{aligned}
	\end{equation}
\end{proposition}
\begin{proof}
\smartqed
To show that it is well-defined is the core of the proof, and the only one that shall be verified.

Take $v_1 +U$ and $v_2+U$ to be any two elements in $V/U$ of the same equivalence class, thus $v_1-v_2 \in U$. Now for any $a\in A$, the elements $(v_1 + U)\cdot_{V/U} a$ and $(v_2+U)\cdot_{V/U} a$ are of the same equivalence class, because $v_1\cdot_V a - v_2\cdot_V a = (v_1-v_2)\cdot_V a\in U$ as $v_1 - v_2\in U$. Likewise, $\alpha_{V/U}$ is a well-defined morphism, for $\alpha_{V}(v_1)-\alpha_V(v_2) = \alpha_V(v_1-v_2)\in U$.

The rest follows simply.
\qed
\end{proof}

Therefore, there is a well defined modules theory for type II Hom-associative algebras. This gives some hopes that it would be possible to have similar results to B\"ack and Richter~\cite{back_non-associative_2018LangloisRemillard,back_hom-associative_2019LangloisRemillard} up to some technicalities, and to the non-trivial verification that there exist Ore-extensions for type II algebras.

The point of the preceding results is the following proposition. It links the categories of modules of algebras with involution and Hom-associative algebras of type II using proposition~\ref{prop:YautwisttypeIILangloisRemillard}.

\begin{proposition}\label{prop:equivalenceofmodulescategoriesLangloisRemillard}
Let $(A,\iota)$ be an $\iota$-algebra. There is a faithful functor $F$ going from the category of left modules of $(A,\iota)$ to the category of right modules of $(A,\circledast, \iota)$ given on objects by
\begin{equation}
\begin{aligned}
F: {  _{A,\iota}\mathsf{Mod}} & \longrightarrow \mathsf{Mod}_{A,\circledast,\iota}\\
M &\longmapsto (M,\cdot_M,\mathrm{id}),\\
\end{aligned}
\end{equation} 
with the action $\cdot_M: M\times A \to M$ given by $m\cdot_M a = \iota(a)m$, and on morphisms by
\begin{equation}
	\begin{aligned}
		F: \operatorname{Hom}_{(A,\iota)}(M,N) &\longrightarrow \operatorname{Hom}_{(A,\circledast,\iota)}(F(M),F(N))\\
		\phi: M\to N & \longmapsto F(\phi): F(M)\to F(N),
	\end{aligned}
\end{equation}
with $F(\phi)(m) = \phi(m)$.
\end{proposition}
\begin{proof}
\smartqed
The functoriality of the proposed $F$ must be verified. Let $M$ be a left module of $(A,\iota)$. That $F(M)$ is a right module of the Hom-associative algebra $(A,\circledast,\iota)$, with the operation $\circledast: A\to A$ of proposition~\ref{prop:YautwisttypeIILangloisRemillard}, necessitates the respect of condition~\eqref{eq:homoduleIILangloisRemillard}. Let $m\in M$ and $a,b\in A$. Taking  $\alpha_A=\iota$ and $\alpha_V = \mathrm{id}$ results on one hand in
\begin{equation*}
	m\cdot_M \iota(a\circledast b) = m\cdot_M \iota^2(ab) = \iota^3(ab)\cdot m = \iota(b)\cdot (\iota(a)\cdot m)
\end{equation*}
and on the other hand in
\begin{equation*}
	\mathrm{id}(m\cdot_M a)\cdot_M b = (m \cdot_M a)\cdot_M b = (\iota(a)\cdot m)\cdot_M b = \iota(b)\cdot(\iota(a)\cdot m).
\end{equation*}
Thus $(M,\cdot_M,\mathrm{id})$ is a right Hom-module.

Let $\phi:M\to N$ be a morphism of left $(A,\iota)$-modules. Then $F(\phi)$ is a morphism of Hom-modules because it respects equations~\eqref{eq:morphismhommoduleIILangloisRemillard}. Let $a\in A$ and $m\in M$. Then
\begin{align*}
	F(\phi)(m\cdot_M a) &= \phi(\iota(a)\cdot m) = \iota(a)\cdot \phi(m) = \phi(m)\cdot_N a,\\
	F(\phi)(\mathrm{id}_M(m)) &= F(\phi)(m) = \mathrm{id}_N(F(\phi)(m)).
\end{align*}
The functor respects the composition of morphisms directly from its definition.

There is thus a well-defined functor $F$. It remains to prove that it is faithful. Let $\phi,\psi : M\to N$ be two morphisms of left $(A,\iota)$-modules. If $F(\phi) = F(\psi)$, then for $m\in M$
\begin{align*}
F(\phi)(m)- F(\psi)(m) &= \phi(m)-\psi(m) = 0.
\end{align*}
Hence $\phi = \psi$ and the application $\phi\mapsto F(\phi)$ is injective, proving the faithfulness of $F$.
\qed
\end{proof}

From this proof, we have that the representation theory of Hom-associative algebras of type II contains a copy of the representation theory of $\iota$-algebras.

\subsection{Example: Temperley-Lieb algebras}

As an example of the past subsections, we will apply the functor of proposition~\ref{prop:YautwisttypeIILangloisRemillard} on the Temperley-Lieb algebra $\mathsf{TL}_4(q+q^{-1})$ and study what happens to its representation theory. Temperley-Lieb algebras are well studied algebras (see the survey~\cite{RSALangloisRemillard} for details and reference therein) that are useful in describing scaling limit for conformal field theories and in knot theory. There are two main way to define them. First, they can be seen as a quotient of a Hecke algebra of type A: the Temperley-Lieb algebra of rank $n$ and of parameter $q\in \mathbb C\backslash\{0\}$ is the associative $\mathbb C$-algebra generated by $n-1$ elements $e_1,\dots,e_{n-1}$, a unit $\mathrm{id}$ and  the relations:
\begin{gather}
\mathrm{id}e_i = e_i\mathrm{id} = e_i; \quad e_i^2 = (q+q^{-1})e_i; \quad e_ie_j = e_je_i, \ (|i-j|>1);\\
e_ie_{i+1}e_i = e_i, \ (1\leq i\leq n-1); \quad e_{i}e_{i-1}e_i = e_i,\ (2\leq i\leq n-1).
\end{gather}

Its dimension is given by the Catalan number
\begin{equation}
\dim \mathsf{TL}_n(q+q^{-1}) = C_n = {1\over n+1} \binom{2n}{n}.
\end{equation}

The $C_4 = 14$ elements of $\mathsf{TL}_{4}(q+q^{-1})$ are given by:
\begin{equation}\label{eq:baseTLQuatregenLangloisRemillard}
\begin{gathered}
\mathrm{id},\\
e_1,\ e_1e_2,\ e_1e_2e_3,\\
e_2,\ e_2e_1,\ e_2e_3,\\
e_3,\ e_3e_2, \ e_3e_2e_1,\\
e_1e_3,\ e_1e_3e_2,\ e_2e_1e_3,\ e_2e_1e_3e_2.
\end{gathered}
\end{equation}

The other way to define the Temperley-Lieb algebra of rank $n$ is via diagrammatic interpretation. A $n$-diagram is a diagram drawn in a rectangle with $n$ points in its left side and $n$ points in its right side all of the $2n$ linked together without crossing. Two diagrams are identified if they differ only by an isotopy. In this interpretation the elements of the algebra are formal $\mathbb C$-linear combinations of $n$-diagrams, and the multiplication is given by concatenation and replacing each of the created closed loops by a factor $q+q^{-1} = [2]_q$. It is an associative unital algebra.

The 14 diagrams giving a vector space basis of $\mathsf{TL}_4(q+q^{-1})$ are given below, ordered by the number of arcs on the same side (the order is the same as~\eqref{eq:baseTLQuatregenLangloisRemillard}):
\begin{gather}\label{eq:baseTLQuatrediagLangloisRemillard}
\   \begin{tikzpicture}[baseline={(current bounding box.center)},scale=\TLdiagscaleLangloisRemillard]
       \draw [line width=0.3mm] (0,-0.7) -- (0,3+.7) ;
       \draw [line width=0.3mm] (2,-0.7) -- (2,3+.7) ;
        \draw (0,0) -- (2,0);       	
        \draw (0,1) -- (2,1);       	
        \draw (0,2) -- (2,2);       	
        \draw (0,3) -- (2,3);       	
    \end{tikzpicture}\ , \notag \\  
   \begin{tikzpicture}[baseline={(current bounding box.center)},scale=\TLdiagscaleLangloisRemillard]
       \draw [line width=0.3mm] (0,-0.7) -- (0,3+.7) ;
       \draw [line width=0.3mm] (2,-0.7) -- (2,3+.7) ;
        \draw (0,0) -- (2,0);       	
        \draw (0,1) -- (2,1);       	
        \draw (0,2) .. controls (3/4,2) and (3/4,3) .. (0,3);       	
        \draw (2,2) .. controls (5/4,2) and (5/4,3) .. (2,3);
    \end{tikzpicture}\ , \quad 
    \begin{tikzpicture}[baseline={(current bounding box.center)},scale=\TLdiagscaleLangloisRemillard]
       \draw [line width=0.3mm] (0,-0.7) -- (0,3+.7) ;
       \draw [line width=0.3mm] (2,-0.7) -- (2,3+.7) ;
        \draw (0,0) -- (2,0);       	
        \draw (2,2) .. controls (5/4,2) and (5/4,1) .. (2,1);       	
        \draw (0,2) .. controls (3/4,2) and (3/4,3) .. (0,3);       	
        \draw (0,1) .. controls (3/4,1) and (5/4,3) .. (2,3);
    \end{tikzpicture}\ , \quad 
    \begin{tikzpicture}[baseline={(current bounding box.center)},scale=\TLdiagscaleLangloisRemillard]
       \draw [line width=0.3mm] (0,-0.7) -- (0,3+.7) ;
       \draw [line width=0.3mm] (2,-0.7) -- (2,3+.7) ;
        \draw (2,0) .. controls (5/4,0) and (5/4,1) .. (2,1);       	
        \draw (0,2) .. controls (3/4,2) and (3/4,3) .. (0,3);       	
        \draw (0,0) .. controls (3/4,0) and (5/4,2) .. (2,2);
        \draw (0,1) .. controls (3/4,1) and (5/4,3) .. (2,3);
    \end{tikzpicture}\ , \notag\\
    \begin{tikzpicture}[baseline={(current bounding box.center)},scale=\TLdiagscaleLangloisRemillard]
    \draw [line width=0.3mm] (0,-0.7) -- (0,3+.7) ;
    \draw [line width=0.3mm] (2,-0.7) -- (2,3+.7) ;
    \draw (0,0) -- (2,0);       	
    \draw (0,3) -- (2,3);
    \draw (0,2) .. controls (3/4,2) and (3/4,1) .. (0,1);       	
    \draw (2,2) .. controls (5/4,2) and (5/4,1) .. (2,1);
    \end{tikzpicture}\ , \quad 
    \begin{tikzpicture}[baseline={(current bounding box.center)},scale=\TLdiagscaleLangloisRemillard]
       \draw [line width=0.3mm] (0,-0.7) -- (0,3+.7) ;
       \draw [line width=0.3mm] (2,-0.7) -- (2,3+.7) ;
        \draw (0,0) -- (2,0);       	
        \draw (0,3) .. controls (3/4,3) and (5/4,1) .. (2,1);
        \draw (0,2) .. controls (3/4,2) and (3/4,1) .. (0,1);       	
        \draw (2,2) .. controls (5/4,2) and (5/4,3) .. (2,3);
    \end{tikzpicture}\ , \quad 
    \begin{tikzpicture}[baseline={(current bounding box.center)},scale=\TLdiagscaleLangloisRemillard]
       \draw [line width=0.3mm] (0,-0.7) -- (0,3+.7) ;
       \draw [line width=0.3mm] (2,-0.7) -- (2,3+.7) ;
        \draw (0,3) -- (2,3);
        \draw (0,0) .. controls (3/4,0) and (5/4,2) .. (2,2);
        \draw (0,2) .. controls (3/4,2) and (3/4,1) .. (0,1);       	
        \draw (2,1) .. controls (5/4,1) and (5/4,0) .. (2,0);
    \end{tikzpicture}\ , \\ 
    \begin{tikzpicture}[baseline={(current bounding box.center)},scale=\TLdiagscaleLangloisRemillard]
    \draw [line width=0.3mm] (0,-0.7) -- (0,3+.7) ;
    \draw [line width=0.3mm] (2,-0.7) -- (2,3+.7) ;
    \draw (0,3) -- (2,3);
    \draw (0,2) -- (2,2);       	
    \draw (0,1) .. controls (3/4,1) and (3/4,0) .. (0,0);       	
    \draw (2,1) .. controls (5/4,1) and (5/4,0) .. (2,0);
    \end{tikzpicture}\ , \quad 
    \begin{tikzpicture}[baseline={(current bounding box.center)},scale=\TLdiagscaleLangloisRemillard]
       \draw [line width=0.3mm] (0,-0.7) -- (0,3+.7) ;
       \draw [line width=0.3mm] (2,-0.7) -- (2,3+.7) ;
        \draw (0,3) -- (2,3);
        \draw (0,2) .. controls (3/4,2) and (5/4,0) .. (2,0);       	
        \draw (0,1) .. controls (3/4,1) and (3/4,0) .. (0,0);       	
        \draw (2,2) .. controls (5/4,2) and (5/4,1) .. (2,1);
    \end{tikzpicture}\ , \quad 
     \begin{tikzpicture}[baseline={(current bounding box.center)},scale=\TLdiagscaleLangloisRemillard]
     \draw [line width=0.3mm] (0,-0.7) -- (0,3+.7) ;
     \draw [line width=0.3mm] (2,-0.7) -- (2,3+.7) ;
     \draw (0,3) .. controls (3/4,3) and (5/4,1) .. (2,1);
     \draw (0,2) .. controls (3/4,2) and (5/4,0) .. (2,0);       	
     \draw (0,1) .. controls (3/4,1) and (3/4,0) .. (0,0);       	
     \draw (2,2) .. controls (5/4,2) and (5/4,3) .. (2,3);
     \end{tikzpicture}\ , \notag \\
\  \begin{tikzpicture}[baseline={(current bounding box.center)},scale=\TLdiagscaleLangloisRemillard]
       \draw [line width=0.3mm] (0,-0.7) -- (0,3+.7) ;
       \draw [line width=0.3mm] (2,-0.7) -- (2,3+.7) ;
        \draw (0,0) .. controls (3/4,0) and (3/4,1) .. (0,1);       	
        \draw (0,2) .. controls (3/4,2) and (3/4,3) .. (0,3);       	
        \draw (2,0) .. controls (5/4,0) and (5/4,1) .. (2,1);
        \draw (2,2) .. controls (5/4,2) and (5/4,3) .. (2,3);
    \end{tikzpicture}\ , \quad 
\begin{tikzpicture}[baseline={(current bounding box.center)},scale=\TLdiagscaleLangloisRemillard]
       \draw [line width=0.3mm] (0,-0.7) -- (0,3+.7) ;
       \draw [line width=0.3mm] (2,-0.7) -- (2,3+.7) ;
        \draw (0,0) .. controls (3/4,0) and (3/4,1) .. (0,1);       	
        \draw (0,2) .. controls (3/4,2) and (3/4,3) .. (0,3);       	
        \draw (2,1) .. controls (1.25,1) and (1.25,2) .. (2,2);
       	\draw (2,0) .. controls (1,0) and (1,3) .. (2,3);
    \end{tikzpicture}\ , \quad
\begin{tikzpicture}[baseline={(current bounding box.center)},scale=\TLdiagscaleLangloisRemillard]
       \draw [line width=0.3mm] (0,-0.7) -- (0,3+.7) ;
       \draw [line width=0.3mm] (2,-0.7) -- (2,3+.7) ;
       	\draw (0,0) .. controls (1,0) and (1,3) .. (0,3);
        \draw (0,1) .. controls (3/4,1) and (3/4,2) .. (0,2);
        \draw (2,0) .. controls (5/4,0) and (5/4,1) .. (2,1);
        \draw (2,2) .. controls (5/4,2) and (5/4,3) .. (2,3);
    \end{tikzpicture}\ , \quad
    \begin{tikzpicture}[baseline={(current bounding box.center)},scale=\TLdiagscaleLangloisRemillard]
       \draw [line width=0.3mm] (0,-0.7) -- (0,3+.7) ;
       \draw [line width=0.3mm] (2,-0.7) -- (2,3+.7) ;
       	\draw (0,0) .. controls (1,0) and (1,3) .. (0,3);
        \draw (0,1) .. controls (3/4,1) and (3/4,2) .. (0,2);
        \draw (2,1) .. controls (1.25,1) and (1.25,2) .. (2,2);
       	\draw (2,0) .. controls (1,0) and (1,3) .. (2,3);
    \end{tikzpicture}\ .\notag
\end{gather} 

The identification between the two definitions follows from the morphism defined by
\begin{align}
\mathrm{id} &\longmapsto \begin{tikzpicture}[baseline={(current bounding box.center)},scale=\TLdiagscaleLangloisRemillard]
       \draw [line width=0.3mm] (0,-0.7) -- (0,5+.7) ;
       \draw [line width=0.3mm] (2,-0.7) -- (2,5+.7) ;
        \draw (0,0) -- (2,0);       	
        \draw (1,1.3) node{ \vdots};
        \draw (0,2) -- (2,2);       	
        \draw (0,3) -- (2,3);
			\node at (1,4.3) {\vdots};
			\draw (0,5) -- (2,5);       	
    \end{tikzpicture}\ , &  e_i &\longmapsto \begin{tikzpicture}[baseline={(current bounding box.center)},scale=\TLdiagscaleLangloisRemillard]
       	\draw [line width=0.3mm] (0,-0.7) -- (0,7 +.7) ;
       	\draw [line width=0.3mm] (2,-0.7) -- (2,7 +.7) ;
        	\draw (0,0) -- (2,0);       	
			\node at (1,1.3) {\vdots};
        	\draw (0,2.25) -- (2,2.25);       	
			\draw (0,4) .. controls (3/4,4) and (3/4,3) .. (0,3);       	
       	\draw (2,4) .. controls (5/4,4) and (5/4,3) .. (2,3);      	
        	\draw (0,4.75) -- (2,4.75);
			\node at (1,6.2) {\vdots};
			\draw (0,7) -- (2,7);      
			\node  at (-0.75,4.1) {$i$};
			\node  at (-1.35,3) {$i+1$};
    \end{tikzpicture}\ .
\end{align}

It is easy to see that the diagrammatic algebra respect the  relations. For example here is the verification of $e_1^2 = (q+q^{-1})e_1$, $e_1e_3=e_3e_1$ and $e_2e_3e_2=e_2$  in $\mathsf{TL}_{4}(q+q^{-1})$:
\begin{gather*}
e_1e_1 \longmapsto\
\begin{tikzpicture}[baseline={(current bounding box.center)},scale=\TLdiagscaleLangloisRemillard]
       \draw [line width=0.3mm] (0,-0.7) -- (0,3+.7) ;
       \draw [line width=0.3mm] (2,-0.7) -- (2,3+.7) ;
        \draw (0,0) -- (2,0);       	
        \draw (0,1) -- (2,1);       	
        \draw (0,2) .. controls (3/4,2) and (3/4,3) .. (0,3);       	
        \draw (2,2) .. controls (5/4,2) and (5/4,3) .. (2,3);
    \end{tikzpicture}\begin{tikzpicture}[baseline={(current bounding box.center)},scale=\TLdiagscaleLangloisRemillard]
       \draw [line width=0.3mm] (0,-0.7) -- (0,3+.7) ;
       \draw [line width=0.3mm] (2,-0.7) -- (2,3+.7) ;
        \draw (0,0) -- (2,0);       	
        \draw (0,1) -- (2,1);       	
        \draw (0,2) .. controls (3/4,2) and (3/4,3) .. (0,3);       	
        \draw (2,2) .. controls (5/4,2) and (5/4,3) .. (2,3);
    \end{tikzpicture}\ = \ (q+q^{-1})\; \begin{tikzpicture}[baseline={(current bounding box.center)},scale=\TLdiagscaleLangloisRemillard]
       \draw [line width=0.3mm] (0,-0.7) -- (0,3+.7) ;
       \draw [line width=0.3mm] (2,-0.7) -- (2,3+.7) ;
        \draw (0,0) -- (2,0);       	
        \draw (0,1) -- (2,1);       	
        \draw (0,2) .. controls (3/4,2) and (3/4,3) .. (0,3);       	
        \draw (2,2) .. controls (5/4,2) and (5/4,3) .. (2,3);
    \end{tikzpicture}\ , \quad 
e_1e_3\longmapsto \
\begin{tikzpicture}[baseline={(current bounding box.center)},scale=\TLdiagscaleLangloisRemillard]
       \draw [line width=0.3mm] (0,-0.7) -- (0,3+.7) ;
       \draw [line width=0.3mm] (2,-0.7) -- (2,3+.7) ;
        \draw (0,0) -- (2,0);       	
        \draw (0,1) -- (2,1);       	
        \draw (0,2) .. controls (3/4,2) and (3/4,3) .. (0,3);       	
        \draw (2,2) .. controls (5/4,2) and (5/4,3) .. (2,3);
    \end{tikzpicture}    \begin{tikzpicture}[baseline={(current bounding box.center)},scale=\TLdiagscaleLangloisRemillard]
       \draw [line width=0.3mm] (0,-0.7) -- (0,3+.7) ;
       \draw [line width=0.3mm] (2,-0.7) -- (2,3+.7) ;
        \draw (0,3) -- (2,3);
        \draw (0,2) -- (2,2);       	
        \draw (0,1) .. controls (3/4,1) and (3/4,0) .. (0,0);       	
        \draw (2,1) .. controls (5/4,1) and (5/4,0) .. (2,0);
    \end{tikzpicture}
\ = \ 
    \begin{tikzpicture}[baseline={(current bounding box.center)},scale=\TLdiagscaleLangloisRemillard]
       \draw [line width=0.3mm] (0,-0.7) -- (0,3+.7) ;
       \draw [line width=0.3mm] (2,-0.7) -- (2,3+.7) ;
        \draw (0,3) -- (2,3);
        \draw (0,2) -- (2,2);       	
        \draw (0,1) .. controls (3/4,1) and (3/4,0) .. (0,0);       	
        \draw (2,1) .. controls (5/4,1) and (5/4,0) .. (2,0);
    \end{tikzpicture}\begin{tikzpicture}[baseline={(current bounding box.center)},scale=\TLdiagscaleLangloisRemillard]
       \draw [line width=0.3mm] (0,-0.7) -- (0,3+.7) ;
       \draw [line width=0.3mm] (2,-0.7) -- (2,3+.7) ;
        \draw (0,0) -- (2,0);       	
        \draw (0,1) -- (2,1);       	
        \draw (0,2) .. controls (3/4,2) and (3/4,3) .. (0,3);       	
        \draw (2,2) .. controls (5/4,2) and (5/4,3) .. (2,3);
    \end{tikzpicture}\ ,\\
    e_2e_3e_2\longmapsto \begin{tikzpicture}[baseline={(current bounding box.center)},scale=\TLdiagscaleLangloisRemillard]
       \draw [line width=0.3mm] (0,-0.7) -- (0,3+.7) ;
       \draw [line width=0.3mm] (2,-0.7) -- (2,3+.7) ;
        \draw (0,0) -- (2,0);       	
        \draw (0,3) -- (2,3);
        \draw (0,2) .. controls (3/4,2) and (3/4,1) .. (0,1);       	
        \draw (2,2) .. controls (5/4,2) and (5/4,1) .. (2,1);
    \end{tikzpicture}\begin{tikzpicture}[baseline={(current bounding box.center)},scale=\TLdiagscaleLangloisRemillard]
       \draw [line width=0.3mm] (0,-0.7) -- (0,3+.7) ;
       \draw [line width=0.3mm] (2,-0.7) -- (2,3+.7) ;
        \draw (0,3) -- (2,3);
        \draw (0,2) -- (2,2);       	
        \draw (0,1) .. controls (3/4,1) and (3/4,0) .. (0,0);       	
        \draw (2,1) .. controls (5/4,1) and (5/4,0) .. (2,0);
    \end{tikzpicture}\begin{tikzpicture}[baseline={(current bounding box.center)},scale=\TLdiagscaleLangloisRemillard]
       \draw [line width=0.3mm] (0,-0.7) -- (0,3+.7) ;
       \draw [line width=0.3mm] (2,-0.7) -- (2,3+.7) ;
        \draw (0,0) -- (2,0);       	
        \draw (0,3) -- (2,3);
        \draw (0,2) .. controls (3/4,2) and (3/4,1) .. (0,1);       	
        \draw (2,2) .. controls (5/4,2) and (5/4,1) .. (2,1);
    \end{tikzpicture} \ = \ \begin{tikzpicture}[baseline={(current bounding box.center)},scale=\TLdiagscaleLangloisRemillard]
       \draw [line width=0.3mm] (0,-0.7) -- (0,3+.7) ;
       \draw [line width=0.3mm] (2,-0.7) -- (2,3+.7) ;
        \draw (0,0) -- (2,0);       	
        \draw (0,3) -- (2,3);
        \draw (0,2) .. controls (3/4,2) and (3/4,1) .. (0,1);       	
        \draw (2,2) .. controls (5/4,2) and (5/4,1) .. (2,1);
    \end{tikzpicture}\ .
\end{gather*}

The advantage of this presentation is readily shown when exhibiting a cellular basis. For $\mathsf{TL}_4(q+q^{-1})$ take $\Lambda := \{0,1,2\}$, the number of arcs on the same side. For $d\in \Lambda$, let $M(d)$ be the set of left half-diagram with $d$ arcs on the same side; the map $C$ simply combines two half-diagrams with the same amounts of arcs in the only way possible after flipping the second one. For example
\begin{equation*}
\begin{tikzpicture}[baseline={(current bounding box.center)},scale=\TLdiagscaleLangloisRemillard]
\draw [line width=0.3mm] (0,-0.7) -- (0,3+.7) ;
\draw (0,0) .. controls (3/4,0) and (3/4,1) .. (0,1);       	
\draw (0,2) -- (1,2);       	
\draw (0,3) -- (1,3);
\end{tikzpicture}\ , \   \begin{tikzpicture}[baseline={(current bounding box.center)},scale=\TLdiagscaleLangloisRemillard]
\draw [line width=0.3mm] (0,-0.7) -- (0,3+.7) ;     	
\draw (0,2) .. controls (3/4,2) and (3/4,3) .. (0,3);       	
\draw (0,0) -- (1,0);
\draw (0,1) --(1,1);
\end{tikzpicture}\quad \longmapsto \quad \begin{tikzpicture}[baseline={(current bounding box.center)},scale=\TLdiagscaleLangloisRemillard]
\draw [line width=0.3mm] (0,-0.7) -- (0,3+.7) ;
\draw (0,0) .. controls (3/4,0) and (3/4,1) .. (0,1);       	
\draw (0,2) -- (1,2);       	
\draw (0,3) -- (1,3);
\end{tikzpicture}\ \dots  \begin{tikzpicture}[baseline={(current bounding box.center)},scale=\TLdiagscaleLangloisRemillard]
\draw [line width=0.3mm] (0,-0.7) -- (0,3+.7) ;     	
\draw (0,2) .. controls (-3/4,2) and (-3/4,3) .. (0,3);       	
\draw (0,0) -- (-1,0);
\draw (0,1) --(-1,1);
\end{tikzpicture} \quad = \quad \begin{tikzpicture}[baseline={(current bounding box.center)},scale=\TLdiagscaleLangloisRemillard]
\draw [line width=0.3mm] (0,-0.7) -- (0,3+.7) ;
\draw [line width=0.3mm] (2,-0.7) -- (2,3+.7) ;
\draw (0,3) .. controls (3/4,3) and (5/4,1) .. (2,1);
\draw (0,2) .. controls (3/4,2) and (5/4,0) .. (2,0);       	
\draw (0,1) .. controls (3/4,1) and (3/4,0) .. (0,0);       	
\draw (2,2) .. controls (5/4,2) and (5/4,3) .. (2,3);
\end{tikzpicture}\ .
\end{equation*}
The anti-involution $\iota$ is simply the reflection of diagrams. It can also be defined as the only anti-endomorphism that leaves invariant the generators of the algebra. For example
\begin{equation*}
\iota\left(\begin{tikzpicture}[baseline={(current bounding box.center)},scale=\TLdiagscaleLangloisRemillard]
\draw [line width=0.3mm] (0,-0.7) -- (0,3+.7) ;
\draw [line width=0.3mm] (2,-0.7) -- (2,3+.7) ;
\draw (0,3) .. controls (3/4,3) and (5/4,1) .. (2,1);
\draw (0,2) .. controls (3/4,2) and (5/4,0) .. (2,0);       	
\draw (0,1) .. controls (3/4,1) and (3/4,0) .. (0,0);       	
\draw (2,2) .. controls (5/4,2) and (5/4,3) .. (2,3);
\end{tikzpicture} \right) \ = \  \begin{tikzpicture}[baseline={(current bounding box.center)},scale=\TLdiagscaleLangloisRemillard]
\draw [line width=0.3mm] (0,-0.7) -- (0,3+.7) ;
\draw [line width=0.3mm] (2,-0.7) -- (2,3+.7) ;
\draw (2,0) .. controls (5/4,0) and (5/4,1) .. (2,1);       	
\draw (0,2) .. controls (3/4,2) and (3/4,3) .. (0,3);       	
\draw (0,0) .. controls (3/4,0) and (5/4,2) .. (2,2);
\draw (0,1) .. controls (3/4,1) and (5/4,3) .. (2,3);
\end{tikzpicture}\ .
\end{equation*}

Constructing the diagrams from half-diagrams is an injective process and all possible cases are covered as $\Lambda$ contains all the possible number of arcs, thus the image of $C$ is a basis of $\mathsf{TL}_4(q+q^{-1})$. Axiom~\eqref{eq:axiomcellularityinvolutionLangloisRemillard} is satisfied as flipping one diagram will indeed simply switch the place of the two half-diagrams, and axiom~\eqref{eq:axiomcellularityLangloisRemillard} amounts to the statement: ``arcs can only be created, never destroyed.''

There are three cell modules for $\mathsf{TL}_4(q+q^{-1})$: $\mathsf{C}_0$, $\mathsf C_1$ and $\mathsf{C}_2$ with respective basis given by:
\begin{equation*}
\begin{aligned}
\mathfrak{B}_0 &= \left\{ \begin{tikzpicture}[baseline={(current bounding box.center)},scale=\TLdiagscaleLangloisRemillard]
\draw [line width=0.3mm] (0,-0.7) -- (0,3+.7) ;
\draw (0,0) -- (1,0);       	
\draw (0,1) -- (1,1);       	
\draw (0,2) -- (1,2);       	
\draw (0,3) -- (1,3);       	
\end{tikzpicture}\right\}, & \mathfrak{B}_1 &= \left\{\begin{tikzpicture}[baseline={(current bounding box.center)},scale=\TLdiagscaleLangloisRemillard]
\draw [line width=0.3mm] (0,-0.7) -- (0,3+.7) ;     	
\draw (0,2) .. controls (3/4,2) and (3/4,3) .. (0,3);       	
\draw (0,0) -- (1,0);
\draw (0,1) --(1,1);
\end{tikzpicture} \ , \   
\begin{tikzpicture}[baseline={(current bounding box.center)},scale=\TLdiagscaleLangloisRemillard]
\draw [line width=0.3mm] (0,-0.7) -- (0,3+.7) ;     	
\draw (0,2) .. controls (3/4,2) and (3/4,1) .. (0,1);       	
\draw (0,0) -- (1,0);
\draw (0,3) --(1,3);
\end{tikzpicture}\ , \ \begin{tikzpicture}[baseline={(current bounding box.center)},scale=\TLdiagscaleLangloisRemillard]
\draw [line width=0.3mm] (0,-0.7) -- (0,3+.7) ;
\draw (0,0) .. controls (3/4,0) and (3/4,1) .. (0,1);       	
\draw (0,2) -- (1,2);       	
\draw (0,3) -- (1,3);
\end{tikzpicture}  \right\}, & \mathfrak{B}_2 &= \left\{ \begin{tikzpicture}[baseline={(current bounding box.center)},scale=\TLdiagscaleLangloisRemillard]
\draw [line width=0.3mm] (0,-0.7) -- (0,3+.7) ;     	
\draw (0,2) .. controls (3/4,2) and (3/4,3) .. (0,3);       	
\draw (0,0) .. controls (3/4,0) and (3/4,1) .. (0,1);
\end{tikzpicture}\ , \ \begin{tikzpicture}[baseline={(current bounding box.center)},scale=\TLdiagscaleLangloisRemillard]
\draw [line width=0.3mm] (0,-0.7) -- (0,3+.7) ;     	
\draw (0,2) .. controls (3/4,2) and (3/4,1) .. (0,1);       	
\draw (0,0) .. controls (1,0) and (1,3) .. (0,3);
\end{tikzpicture}  \right\}.
\end{aligned}
\end{equation*} 

The action is also given by concatenation with the extra rules that whenever a new arc is created, the result is zero. For example
\begin{equation}
\begin{tikzpicture}[baseline={(current bounding box.center)},scale=\TLdiagscaleLangloisRemillard]
\draw [line width=0.3mm] (0,-0.7) -- (0,3+.7) ;
\draw [line width=0.3mm] (2,-0.7) -- (2,3+.7) ;
\draw (2,0) .. controls (5/4,0) and (5/4,1) .. (2,1);       	
\draw (0,2) .. controls (3/4,2) and (3/4,3) .. (0,3);       	
\draw (0,0) .. controls (3/4,0) and (5/4,2) .. (2,2);
\draw (0,1) .. controls (3/4,1) and (5/4,3) .. (2,3);
\end{tikzpicture}\begin{tikzpicture}[baseline={(current bounding box.center)},scale=\TLdiagscaleLangloisRemillard]
\draw [line width=0.3mm] (0,-0.7) -- (0,3+.7) ;     	
\draw (0,2) .. controls (3/4,2) and (3/4,3) .. (0,3);       	
\draw (0,0) -- (1,0);
\draw (0,1) --(1,1);
\end{tikzpicture}  = 0, \qquad \begin{tikzpicture}[baseline={(current bounding box.center)},scale=\TLdiagscaleLangloisRemillard]
\draw [line width=0.3mm] (0,-0.7) -- (0,3+.7) ;
\draw [line width=0.3mm] (2,-0.7) -- (2,3+.7) ;
\draw (2,0) .. controls (5/4,0) and (5/4,1) .. (2,1);       	
\draw (0,2) .. controls (3/4,2) and (3/4,3) .. (0,3);       	
\draw (0,0) .. controls (3/4,0) and (5/4,2) .. (2,2);
\draw (0,1) .. controls (3/4,1) and (5/4,3) .. (2,3);
\end{tikzpicture}\begin{tikzpicture}[baseline={(current bounding box.center)},scale=\TLdiagscaleLangloisRemillard]
\draw [line width=0.3mm] (0,-0.7) -- (0,3+.7) ;
\draw (0,0) .. controls (3/4,0) and (3/4,1) .. (0,1);       	
\draw (0,2) -- (1,2);       	
\draw (0,3) -- (1,3);
\end{tikzpicture}\ = (q+q^{-1})\ \begin{tikzpicture}[baseline={(current bounding box.center)},scale=\TLdiagscaleLangloisRemillard]
\draw [line width=0.3mm] (0,-0.7) -- (0,3+.7) ;
\draw (0,0) .. controls (3/4,0) and (3/4,1) .. (0,1);       	
\draw (0,2) -- (1,2);       	
\draw (0,3) -- (1,3);
\end{tikzpicture}\ .
\end{equation}

When $q$ is not a root of unity, the Temperley-Lieb algebra is semisimple and decomposes as a module on itself by the Wedderburn theorem in the direct sum:
\begin{equation}
\mathsf{TL}_4(q+q^{-1}) = \bigoplus_{d\in \Lambda} \dim( \mathsf{C}_d) \mathsf{C}_d.
\end{equation}

After applying proposition~\ref{prop:YautwisttypeIILangloisRemillard}, the new multiplication of the Hom-associative algebra of type II $(\mathsf{TL}_4(q+q^{-1}),\circledast,\iota)$ simply flips the result of the old. For example
\begin{equation*}
\begin{tikzpicture}[baseline={(current bounding box.center)},scale=\TLdiagscaleLangloisRemillard]
\draw [line width=0.3mm] (0,-0.7) -- (0,3+.7) ;
\draw [line width=0.3mm] (2,-0.7) -- (2,3+.7) ;
\draw (0,0) -- (2,0);       	
\draw (0,1) -- (2,1);       	
\draw (0,2) .. controls (3/4,2) and (3/4,3) .. (0,3);       	
\draw (2,2) .. controls (5/4,2) and (5/4,3) .. (2,3);
\end{tikzpicture} \ \circledast \ \begin{tikzpicture}[baseline={(current bounding box.center)},scale=\TLdiagscaleLangloisRemillard]
\draw [line width=0.3mm] (0,-0.7) -- (0,3+.7) ;
\draw [line width=0.3mm] (2,-0.7) -- (2,3+.7) ;
\draw (0,0) -- (2,0);       	
\draw (0,3) -- (2,3);       	
\draw (0,2) .. controls (3/4,2) and (3/4,1) .. (0,1);       	
\draw (2,2) .. controls (5/4,2) and (5/4,1) .. (2,1);
\end{tikzpicture} \ = \ \iota \left(\, \begin{tikzpicture}[baseline={(current bounding box.center)},scale=\TLdiagscaleLangloisRemillard]
\draw [line width=0.3mm] (0,-0.7) -- (0,3+.7) ;
\draw [line width=0.3mm] (2,-0.7) -- (2,3+.7) ;
\draw (0,0) -- (2,0);       	
\draw (0,1) .. controls (3/4,1) and (5/4,3) .. (2,3);
\draw (0,2) .. controls (3/4,2) and (3/4,3) .. (0,3);       	
\draw (2,2) .. controls (5/4,2) and (5/4,1) .. (2,1);
\end{tikzpicture} \, \right) \ = \ \begin{tikzpicture}[baseline={(current bounding box.center)},scale=\TLdiagscaleLangloisRemillard]
\draw [line width=0.3mm] (0,-0.7) -- (0,3+.7) ;
\draw [line width=0.3mm] (2,-0.7) -- (2,3+.7) ;
\draw (0,0) -- (2,0);       	
\draw (0,3) .. controls (3/4,3) and (5/4,1) .. (2,1);
\draw (0,2) .. controls (3/4,2) and (3/4,1) .. (0,1);       	
\draw (2,2) .. controls (5/4,2) and (5/4,3) .. (2,3);
\end{tikzpicture}\ .
\end{equation*}

Interestingly, the action on module changes in a very natural way in this diagrammatic setting. Sending the left module $\mathsf{C}$ to a right Hom-module by $m\cdot a := \iota(a) m$ is portrayed in diagrammatic form simply by flipping the orientation of the half-diagram and keeping the natural action by concatenation.

The new bases of the right cell modules $\mathsf{C}_0$, $\mathsf{C}_1$ and $\mathsf{C}_2$ (with $\alpha_{\mathsf{C}} = \mathrm{id}_{\mathsf C}$) of the Hom-associative algebra of type II $(\mathsf{TL}_4(q+q^{-1}), \circledast,\iota)$ are given by
\begin{equation*}
\begin{aligned}
\mathfrak{B}'_0 &= \left\{ \begin{tikzpicture}[baseline={(current bounding box.center)},scale=\TLdiagscaleLangloisRemillard]
\draw [line width=0.3mm] (0,-0.7) -- (0,3+.7) ;
\draw (0,0) -- (-1,0);       	
\draw (0,1) -- (-1,1);       	
\draw (0,2) -- (-1,2);       	
\draw (0,3) -- (-1,3);       	
\end{tikzpicture}\right\}, & 
\mathfrak{B}_1' &= \left\{\begin{tikzpicture}[baseline={(current bounding box.center)},scale=\TLdiagscaleLangloisRemillard]
\draw [line width=0.3mm] (0,-0.7) -- (0,3+.7) ;     	
\draw (0,2) .. controls (-3/4,2) and (-3/4,3) .. (0,3);       	
\draw (0,0) -- (-1,0);
\draw (0,1) --(-1,1);
\end{tikzpicture} \ , \   
\begin{tikzpicture}[baseline={(current bounding box.center)},scale=\TLdiagscaleLangloisRemillard]
\draw [line width=0.3mm] (0,-0.7) -- (0,3+.7) ;     	
\draw (0,2) .. controls (-3/4,2) and (-3/4,1) .. (0,1);       	
\draw (0,0) -- (-1,0);
\draw (0,3) --(-1,3);
\end{tikzpicture}\ , \ \begin{tikzpicture}[baseline={(current bounding box.center)},scale=\TLdiagscaleLangloisRemillard]
\draw [line width=0.3mm] (0,-0.7) -- (0,3+.7) ;
\draw (0,0) .. controls (-3/4,0) and (-3/4,1) .. (0,1);       	
\draw (0,2) -- (-1,2);       	
\draw (0,3) -- (-1,3);
\end{tikzpicture}  \right\}, & 
\mathfrak{B}_2' &= \left\{ \begin{tikzpicture}[baseline={(current bounding box.center)},scale=\TLdiagscaleLangloisRemillard]
\draw [line width=0.3mm] (0,-0.7) -- (0,3+.7) ;     	
\draw (0,2) .. controls (-3/4,2) and (-3/4,3) .. (0,3);       	
\draw (0,0) .. controls (-3/4,0) and (-3/4,1) .. (0,1);
\end{tikzpicture}\ , \ \begin{tikzpicture}[baseline={(current bounding box.center)},scale=\TLdiagscaleLangloisRemillard]
\draw [line width=0.3mm] (0,-0.7) -- (0,3+.7) ;     	
\draw (0,2) .. controls (-3/4,2) and (-3/4,1) .. (0,1);       	
\draw (0,0) .. controls (-1,0) and (-1,3) .. (0,3);
\end{tikzpicture}  \right\}.
\end{aligned}
\end{equation*} 

The action is given simply by concatenation diagrammatically, which amounts formally to the functor $F$ of proposition~\ref{prop:equivalenceofmodulescategoriesLangloisRemillard}. For example:
\begin{gather*}
\begin{tikzpicture}[baseline={(current bounding box.center)},scale=\TLdiagscaleLangloisRemillard]
\draw [line width=0.3mm] (0,-0.7) -- (0,3+.7) ;     	
\draw (0,2) .. controls (-3/4,2) and (-3/4,3) .. (0,3);       	
\draw (0,0) -- (-1,0);
\draw (0,1) --(-1,1);
\end{tikzpicture}\begin{tikzpicture}[baseline={(current bounding box.center)},scale=\TLdiagscaleLangloisRemillard]
\draw [line width=0.3mm] (0,-0.7) -- (0,3+.7) ;
\draw [line width=0.3mm] (2,-0.7) -- (2,3+.7) ;
\draw (0,0) -- (2,0);       	
\draw (0,3) .. controls (3/4,3) and (5/4,1) .. (2,1);
\draw (0,2) .. controls (3/4,2) and (3/4,1) .. (0,1);       	
\draw (2,2) .. controls (5/4,2) and (5/4,3) .. (2,3);
\end{tikzpicture} \ = \ \begin{tikzpicture}[baseline={(current bounding box.center)},scale=\TLdiagscaleLangloisRemillard]
\draw [line width=0.3mm] (0,-0.7) -- (0,3+.7) ;     	
\draw (0,2) .. controls (-3/4,2) and (-3/4,3) .. (0,3);       	
\draw (0,0) -- (-1,0);
\draw (0,1) --(-1,1);
\end{tikzpicture}
\intertext{and formally by,}
\begin{tikzpicture}[baseline={(current bounding box.center)},scale=\TLdiagscaleLangloisRemillard]
\draw [line width=0.3mm] (0,-0.7) -- (0,3+.7) ;     	
\draw (0,2) .. controls (3/4,2) and (3/4,3) .. (0,3);       	
\draw (0,0) -- (1,0);
\draw (0,1) --(1,1);
\end{tikzpicture}\ \cdot \ \begin{tikzpicture}[baseline={(current bounding box.center)},scale=\TLdiagscaleLangloisRemillard]
\draw [line width=0.3mm] (0,-0.7) -- (0,3+.7) ;
\draw [line width=0.3mm] (2,-0.7) -- (2,3+.7) ;
\draw (0,0) -- (2,0);       	
\draw (0,3) .. controls (3/4,3) and (5/4,1) .. (2,1);
\draw (0,2) .. controls (3/4,2) and (3/4,1) .. (0,1);       	
\draw (2,2) .. controls (5/4,2) and (5/4,3) .. (2,3);
\end{tikzpicture} \ := \ 
\iota\left(\, \begin{tikzpicture}[baseline={(current bounding box.center)},scale=\TLdiagscaleLangloisRemillard]
\draw [line width=0.3mm] (0,-0.7) -- (0,3+.7) ;
\draw [line width=0.3mm] (2,-0.7) -- (2,3+.7) ;
\draw (0,0) -- (2,0);       	
\draw (0,3) .. controls (3/4,3) and (5/4,1) .. (2,1);
\draw (0,2) .. controls (3/4,2) and (3/4,1) .. (0,1);       	
\draw (2,2) .. controls (5/4,2) and (5/4,3) .. (2,3);
\end{tikzpicture}\,\right)\begin{tikzpicture}[baseline={(current bounding box.center)},scale=\TLdiagscaleLangloisRemillard]
\draw [line width=0.3mm] (0,-0.7) -- (0,3+.7) ;     	
\draw (0,2) .. controls (3/4,2) and (3/4,3) .. (0,3);       	
\draw (0,0) -- (1,0);
\draw (0,1) --(1,1);
\end{tikzpicture}\ = \  \begin{tikzpicture}[baseline={(current bounding box.center)},scale=\TLdiagscaleLangloisRemillard]
\draw [line width=0.3mm] (0,-0.7) -- (0,3+.7) ;
\draw [line width=0.3mm] (2,-0.7) -- (2,3+.7) ;
\draw (0,0) -- (2,0);       	
\draw (0,1) .. controls (3/4,1) and (5/4,3) .. (2,3);
\draw (0,2) .. controls (3/4,2) and (3/4,3) .. (0,3);       	
\draw (2,2) .. controls (5/4,2) and (5/4,1) .. (2,1);
\end{tikzpicture}
\begin{tikzpicture}[baseline={(current bounding box.center)},scale=\TLdiagscaleLangloisRemillard]
\draw [line width=0.3mm] (0,-0.7) -- (0,3+.7) ;     	
\draw (0,2) .. controls (3/4,2) and (3/4,3) .. (0,3);       	
\draw (0,0) -- (1,0);
\draw (0,1) --(1,1);
\end{tikzpicture} = \begin{tikzpicture}[baseline={(current bounding box.center)},scale=\TLdiagscaleLangloisRemillard]
\draw [line width=0.3mm] (0,-0.7) -- (0,3+.7) ;     	
\draw (0,2) .. controls (3/4,2) and (3/4,3) .. (0,3);       	
\draw (0,0) -- (1,0);
\draw (0,1) --(1,1);
\end{tikzpicture}\ .
\end{gather*}

Equation~\ref{eq:homoduleIILangloisRemillard} is also respected in this setting. For example the right-hand side is
\begin{align*}
\begin{tikzpicture}[baseline={(current bounding box.center)},scale=\TLdiagscaleLangloisRemillard]
\draw [line width=0.3mm] (0,-0.7) -- (0,3+.7) ;     	
\draw (0,2) .. controls (-3/4,2) and (-3/4,3) .. (0,3);       	
\draw (0,0) -- (-1,0);
\draw (0,1) --(-1,1);
\end{tikzpicture}\ \cdot \ \iota\left(\begin{tikzpicture}[baseline={(current bounding box.center)},scale=\TLdiagscaleLangloisRemillard]
\draw [line width=0.3mm] (0,-0.7) -- (0,3+.7) ;
\draw [line width=0.3mm] (2,-0.7) -- (2,3+.7) ;
\draw (0,0) -- (2,0);       	
\draw (0,1) -- (2,1);       	
\draw (0,2) .. controls (3/4,2) and (3/4,3) .. (0,3);       	
\draw (2,2) .. controls (5/4,2) and (5/4,3) .. (2,3);
\end{tikzpicture} \ \circledast \ \begin{tikzpicture}[baseline={(current bounding box.center)},scale=\TLdiagscaleLangloisRemillard]
\draw [line width=0.3mm] (0,-0.7) -- (0,3+.7) ;
\draw [line width=0.3mm] (2,-0.7) -- (2,3+.7) ;
\draw (0,0) -- (2,0);       	
\draw (0,3) -- (2,3);       	
\draw (0,2) .. controls (3/4,2) and (3/4,1) .. (0,1);       	
\draw (2,2) .. controls (5/4,2) and (5/4,1) .. (2,1);
\end{tikzpicture}\right) \ = \ 
\begin{tikzpicture}[baseline={(current bounding box.center)},scale=\TLdiagscaleLangloisRemillard]
\draw [line width=0.3mm] (0,-0.7) -- (0,3+.7) ;     	
\draw (0,2) .. controls (-3/4,2) and (-3/4,3) .. (0,3);       	
\draw (0,0) -- (-1,0);
\draw (0,1) --(-1,1);
\end{tikzpicture}\ \cdot \ \iota^2 \left(\, \begin{tikzpicture}[baseline={(current bounding box.center)},scale=\TLdiagscaleLangloisRemillard]
\draw [line width=0.3mm] (0,-0.7) -- (0,3+.7) ;
\draw [line width=0.3mm] (2,-0.7) -- (2,3+.7) ;
\draw (0,0) -- (2,0);       	
\draw (0,1) .. controls (3/4,1) and (5/4,3) .. (2,3);
\draw (0,2) .. controls (3/4,2) and (3/4,3) .. (0,3);       	
\draw (2,2) .. controls (5/4,2) and (5/4,1) .. (2,1);
\end{tikzpicture} \, \right) \ = \ 
\begin{tikzpicture}[baseline={(current bounding box.center)},scale=\TLdiagscaleLangloisRemillard]
\draw [line width=0.3mm] (0,-0.7) -- (0,3+.7) ;     	
\draw (0,2) .. controls (-3/4,2) and (-3/4,3) .. (0,3);       	
\draw (0,0) -- (-1,0);
\draw (0,1) --(-1,1);
\end{tikzpicture}\begin{tikzpicture}[baseline={(current bounding box.center)},scale=\TLdiagscaleLangloisRemillard]
\draw [line width=0.3mm] (0,-0.7) -- (0,3+.7) ;
\draw [line width=0.3mm] (2,-0.7) -- (2,3+.7) ;
\draw (0,0) -- (2,0);       	
\draw (0,1) .. controls (3/4,1) and (5/4,3) .. (2,3);
\draw (0,2) .. controls (3/4,2) and (3/4,3) .. (0,3);       	
\draw (2,2) .. controls (5/4,2) and (5/4,1) .. (2,1);
\end{tikzpicture}\ = \ (q+q^{-1})\ \begin{tikzpicture}[baseline={(current bounding box.center)},scale=\TLdiagscaleLangloisRemillard]
\draw [line width=0.3mm] (0,-0.7) -- (0,3+.7) ;     	
\draw (0,2) .. controls (-3/4,2) and (-3/4,1) .. (0,1);       	
\draw (0,0) -- (-1,0);
\draw (0,3) --(-1,3);
\end{tikzpicture}\ ,
\end{align*}
and the left-hand side is given by
\begin{align*}
\mathrm{id}_{\mathsf{C}_1}\left(\begin{tikzpicture}[baseline={(current bounding box.center)},scale=\TLdiagscaleLangloisRemillard]
\draw [line width=0.3mm] (0,-0.7) -- (0,3+.7) ;     	
\draw (0,2) .. controls (-3/4,2) and (-3/4,3) .. (0,3);       	
\draw (0,0) -- (-1,0);
\draw (0,1) --(-1,1);
\end{tikzpicture}\begin{tikzpicture}[baseline={(current bounding box.center)},scale=\TLdiagscaleLangloisRemillard]
\draw [line width=0.3mm] (0,-0.7) -- (0,3+.7) ;
\draw [line width=0.3mm] (2,-0.7) -- (2,3+.7) ;
\draw (0,0) -- (2,0);       	
\draw (0,1) -- (2,1);       	
\draw (0,2) .. controls (3/4,2) and (3/4,3) .. (0,3);       	
\draw (2,2) .. controls (5/4,2) and (5/4,3) .. (2,3);
\end{tikzpicture}\right) \ \cdot \ \begin{tikzpicture}[baseline={(current bounding box.center)},scale=\TLdiagscaleLangloisRemillard]
\draw [line width=0.3mm] (0,-0.7) -- (0,3+.7) ;
\draw [line width=0.3mm] (2,-0.7) -- (2,3+.7) ;
\draw (0,0) -- (2,0);       	
\draw (0,3) -- (2,3);       	
\draw (0,2) .. controls (3/4,2) and (3/4,1) .. (0,1);       	
\draw (2,2) .. controls (5/4,2) and (5/4,1) .. (2,1);
\end{tikzpicture} \ = \ (q+q^{-1})\ \begin{tikzpicture}[baseline={(current bounding box.center)},scale=\TLdiagscaleLangloisRemillard]
\draw [line width=0.3mm] (0,-0.7) -- (0,3+.7) ;     	
\draw (0,2) .. controls (-3/4,2) and (-3/4,3) .. (0,3);       	
\draw (0,0) -- (-1,0);
\draw (0,1) --(-1,1);
\end{tikzpicture}\begin{tikzpicture}[baseline={(current bounding box.center)},scale=\TLdiagscaleLangloisRemillard]
\draw [line width=0.3mm] (0,-0.7) -- (0,3+.7) ;
\draw [line width=0.3mm] (2,-0.7) -- (2,3+.7) ;
\draw (0,0) -- (2,0);       	
\draw (0,3) -- (2,3);       	
\draw (0,2) .. controls (3/4,2) and (3/4,1) .. (0,1);       	
\draw (2,2) .. controls (5/4,2) and (5/4,1) .. (2,1);
\end{tikzpicture} \ = \ 
(q+q^{-1})\ \begin{tikzpicture}[baseline={(current bounding box.center)},scale=\TLdiagscaleLangloisRemillard]
\draw [line width=0.3mm] (0,-0.7) -- (0,3+.7) ;     	
\draw (0,2) .. controls (-3/4,2) and (-3/4,1) .. (0,1);       	
\draw (0,0) -- (-1,0);
\draw (0,3) --(-1,3);
\end{tikzpicture}\ .
\end{align*}

The algebra $B=(\mathsf{TL}_4(q+q^{-1}),\circledast,\iota)$ keeps its cell filtration~\eqref{eq:baseTLQuatrediagLangloisRemillard}:
\begin{equation}\label{eq:filtrationTLhomLangloisRemillard}
B= \left\langle \
\begin{tikzpicture}[baseline={(current bounding box.center)},scale=\TLdiagscaleLangloisRemillard]
\draw [line width=0.3mm] (0,-0.7) -- (0,3+.7) ;
\draw [line width=0.3mm] (2,-0.7) -- (2,3+.7) ;
\draw (0,0) -- (2,0);       	
\draw (0,1) -- (2,1);       	
\draw (0,2) -- (2,2);       	
\draw (0,3) -- (2,3);       	
\end{tikzpicture}\ \right\rangle \supset \left\langle \
\begin{tikzpicture}[baseline={(current bounding box.center)},scale=\TLdiagscaleLangloisRemillard]
\draw [line width=0.3mm] (0,-0.7) -- (0,3+.7) ;
\draw [line width=0.3mm] (2,-0.7) -- (2,3+.7) ;
\draw (0,0) -- (2,0);       	
\draw (0,1) -- (2,1);       	
\draw (0,2) .. controls (3/4,2) and (3/4,3) .. (0,3);       	
\draw (2,2) .. controls (5/4,2) and (5/4,3) .. (2,3);
\end{tikzpicture}
\ , \ 
\begin{tikzpicture}[baseline={(current bounding box.center)},scale=\TLdiagscaleLangloisRemillard]
\draw [line width=0.3mm] (0,-0.7) -- (0,3+.7) ;
\draw [line width=0.3mm] (2,-0.7) -- (2,3+.7) ;
\draw (0,0) -- (2,0);       	
\draw (0,3) -- (2,3);
\draw (0,2) .. controls (3/4,2) and (3/4,1) .. (0,1);       	
\draw (2,2) .. controls (5/4,2) and (5/4,1) .. (2,1);
\end{tikzpicture}\ , \
\begin{tikzpicture}[baseline={(current bounding box.center)},scale=\TLdiagscaleLangloisRemillard]
\draw [line width=0.3mm] (0,-0.7) -- (0,3+.7) ;
\draw [line width=0.3mm] (2,-0.7) -- (2,3+.7) ;
\draw (0,3) -- (2,3);
\draw (0,2) -- (2,2);       	
\draw (0,1) .. controls (3/4,1) and (3/4,0) .. (0,0);       	
\draw (2,1) .. controls (5/4,1) and (5/4,0) .. (2,0);
\end{tikzpicture}\
 \right\rangle \supset \left\langle 
\  \begin{tikzpicture}[baseline={(current bounding box.center)},scale=\TLdiagscaleLangloisRemillard]
\draw [line width=0.3mm] (0,-0.7) -- (0,3+.7) ;
\draw [line width=0.3mm] (2,-0.7) -- (2,3+.7) ;
\draw (0,0) .. controls (3/4,0) and (3/4,1) .. (0,1);       	
\draw (0,2) .. controls (3/4,2) and (3/4,3) .. (0,3);       	
\draw (2,0) .. controls (5/4,0) and (5/4,1) .. (2,1);
\draw (2,2) .. controls (5/4,2) and (5/4,3) .. (2,3);
\end{tikzpicture}\ , \
\begin{tikzpicture}[baseline={(current bounding box.center)},scale=\TLdiagscaleLangloisRemillard]
\draw [line width=0.3mm] (0,-0.7) -- (0,3+.7) ;
\draw [line width=0.3mm] (2,-0.7) -- (2,3+.7) ;
\draw (0,0) .. controls (1,0) and (1,3) .. (0,3);
\draw (0,1) .. controls (3/4,1) and (3/4,2) .. (0,2);
\draw (2,1) .. controls (1.25,1) and (1.25,2) .. (2,2);
\draw (2,0) .. controls (1,0) and (1,3) .. (2,3);
\end{tikzpicture}\ \right\rangle \supset 0.
\end{equation}

If $q$ is not a root of unity, the algebra is Hom-semisimple. Indeed, variations on the arguments of~\cite{RSALangloisRemillard} let one easily show that each cell module is cyclic and $q$ not being a root of unity implies that each elements of a cell module is a generator, the filtration~\eqref{eq:filtrationTLhomLangloisRemillard} finishes the proof.

It is not surprising as for the Temperley-Lieb algebra, the process of going to Hom-associativity of type II is very similar on the representation theory level to considering the left cell modules as right modules: only one application of $\iota$ separates the two concepts.

Furthermore, for any cellular algebra, the cellular filtration is kept. Applying the faithful functor $F$ to K\"onig's and Xi's definition of cellularity (definition~\ref{dfn:cellularKXLangloisRemillard}) results in the following definition.

\begin{definition}
	Let $(A,\cdot,\alpha)$ be a Hom-associative algebra over an associative, commutative Noetherian integral domain $R$. Assume that there is a anti-involution $\iota$ in $A$. A two-sided Hom-ideal $J$ of $A$ is called a \emph{Hom-cell ideal} if 
	\begin{enumerate}
		\item it is fixed by the anti-involution: $\iota(J) =J$;
		\item there exists a Hom-module $\Delta \subset J$ such that $\Delta$ is finitely generated and free over $R$;
		\item there is an isomorphism of Hom-bimodules $\psi : J \xrightarrow{\ \sim\ } \Delta \otimes_R \iota(\Delta)$.
		\begin{equation}
		\begin{tikzcd}[column sep = large, row sep = large]
			J \arrow[d,"\iota"] \arrow[r, "\psi"] & \Delta \otimes_R
			\iota(\Delta) \arrow[d,"x\otimes y \mapsto \iota(y)\otimes \iota(x)"]\\
			J \arrow[r,"\psi"] & \Delta \otimes_R \iota(\Delta)
		\end{tikzcd}.
		\end{equation}
	\end{enumerate}
The algebra with the anti-involution $\iota$ is called \emph{Hom-cellular} if there is a Hom-$R$-modules decomposition
\begin{equation*}
	A = J_1' \oplus J_2' \oplus \dots \oplus J_n'
\end{equation*}
with $\iota(J_k') = J_k'$ for each $k$ such that setting $J_k = \bigoplus_{l=1}^k J_l'$ gives a chain of Hom-$A$-ideals of $A$
\begin{equation}
	0 = J_0 \subset J_1 \subset \dots \subset J_n = A,
\end{equation}
in which $J_k'$ is a Hom-cell ideal for $A/J_{k-1}$.
\end{definition}

Therefore we see that the functor $F$ preserves the structure for some subfamilies of algebras with anti-involution. Of course this short inquiry only presents arguments for the albeitly trivial case of semisimple cellular in which semisimplicity will be preserved by the faithfulness of the functor and the weaker notion of Hom-semisimplicity, but it hints that further investigation with weaker structures on algebra with anti-involution could preserve a sufficient amount of structure to be studied in the Hom-associativity of type II framework; that it could deform ``enough'' to open new applications is left for further studies.
%
%


\section{Discussion}
\label{sec:discussionLangloisRemillard}
In this independent short section, we will consider an avenue to systematize the study started here. It contains some ideas borrowed from Hellstr\"om, Makhlouf and Silvestrov~\cite{hellstrom_universal_2014LangloisRemillard} about universal algebras and from Yau~\cite{yau_enveloping_2008LangloisRemillard}. Our main interest is the diagrammatic operadic approach that seems fruitful in considering the slight changes in the axiomatic rule that was exemplified here.

That type II associativity arose when one tried to deform a $\iota$-algebra was an unexpected discovery. It justifies the construction of a module theory and consideration of this peculiar type of deformations. It is probably possible to do the same study for other types of Hom-associativity. The goal of this section is hint at a more systematic way to do so.
 
We do not claim to offer the correct way of generalizing these notions. We merely aim to present an interesting piece of material and show an application to Fr\'egier and Gohr hierarchy~\cite{fregier_hom-type_2010LangloisRemillard} where this material proved useful to schematize the proofs.
 
 Take a monoid $M$. Represent its identity map $\mathrm{id}:{M}\to {M}$ and its multiplication $\mu:{M}\times {M}\to {M}$ by the following diagrams:
 	\begin{align}
 		\mathrm{id} &\longrightarrow \tikzdessin{
 		\draw[thick] (0,0) -- (0,1);
 		\fill (0,0) circle(2pt);
 		\fill (0,1) circle(2pt);
 		}, & \mu &\longrightarrow \tikzdessin{
 		\draw[thick] (0,0) .. controls (0,0.5) and (0.5,0.75) .. (0.5,1);
 		\draw[thick] (1,0) .. controls (1,0.5) and (0.5,0.75) .. (0.5,1);
 		\fill (0,0) circle(2pt);
 		\fill (1,0) circle(2pt);
 		\fill (0.5,1) circle(2pt);}.
 	\end{align}

They are read from bottom to top. The diagram of $\mu$ takes two elements and gives back one. The identity condition is implicit here for one can deform the diagram as wanted as long as the topology is left unchanged. To add associativity, there must be a relation equivalent to $(\mu(\mu(a,b),c) =\mu(a,\mu(b,c))$. This is done by:
\begin{align}
	\tikzdessin{
 		\draw[thick] (0,0) .. controls (0,0.5) and (0.5,0.75) .. (0.5,1);
 		\draw[thick] (1,0) .. controls (1,0.5) and (0.5,0.75) .. (0.5,1);
 		\fill (0,0) circle(2pt);
 		\fill (1,0) circle(2pt);
 		\fill (0.5,1) circle(2pt);
 		\draw[thick] (2,0) -- (2,1);
 		\fill (2,0) circle(2pt);
 		\fill (2,1) circle(2pt);
 		\draw[thick] (0.5,1) .. controls (0.5,1.5) and (1.25,1.75) .. (1.25,2);
 		\draw[thick] (2,1) .. controls (2,1.5) and (1.25,1.75) .. (1.25,2);
 		\fill (1.25,2) circle(2pt);
 		} \quad =\quad 
 		\tikzdessin{
 		\draw[thick] (0,0) .. controls (0,0.5) and (-0.5,0.75) .. (-0.5,1);
 		\draw[thick] (-1,0) .. controls (-1,0.5) and (-0.5,0.75) .. (-0.5,1);
 		\fill (0,0) circle(2pt);
 		\fill (-1,0) circle(2pt);
 		\fill (-0.5,1) circle(2pt);
 		\draw[thick] (-2,0) -- (-2,1);
 		\fill (-2,0) circle(2pt);
 		\fill (-2,1) circle(2pt);
 		\draw[thick] (-0.5,1) .. controls (-0.5,1.5) and (-1.25,1.75) .. (-1.25,2);
 		\draw[thick] (-2,1) .. controls (-2,1.5) and (-1.25,1.75) .. (-1.25,2);
 		\fill (-1.25,2) circle(2pt);
 		}\ .
\end{align}

The system is simple for now. To consider Hom-associativity ($\mu(\mu(a,b),\alpha(c))=\mu(\alpha(a),\mu(b,c))$), one must add a diagram for the twisting map $\alpha$ and then use it to obtain a new condition to replace the associativity. It looks like
\begin{align}
	\alpha &\longrightarrow \tikzdessin{
	\draw[thick] (0,0) -- (0,1);
 	\fill (0,0) circle(2pt);
    \fill (0,1) circle(2pt);
 	\fill (0.1,0.4) -- (-0.1,0.4)--(-0.1,0.6) --(0.1,0.6)--cycle;
 	}\ ,\\ \label{eq:homassIdessinLangloisRemillard}
 	\tikzdessin{
 		\draw[thick] (0,0) .. controls (0,0.5) and (0.5,0.75) .. (0.5,1);
 		\draw[thick] (1,0) .. controls (1,0.5) and (0.5,0.75) .. (0.5,1);
 		\fill (0,0) circle(2pt);
 		\fill (1,0) circle(2pt);
 		\fill (0.5,1) circle(2pt);
 		\draw[thick] (2,0) -- (2,1);
 		\fill (2,0) circle(2pt);
 		\fill (2,1) circle(2pt);
 		\fill (2.1,0.4) -- (1.9,0.4)--(1.9,0.6) --(2.1,0.6)--cycle;
 		\draw[thick] (0.5,1) .. controls (0.5,1.5) and (1.25,1.75) .. (1.25,2);
 		\draw[thick] (2,1) .. controls (2,1.5) and (1.25,1.75) .. (1.25,2);
 		\fill (1.25,2) circle(2pt);
 		} \quad &=\quad 
 		\tikzdessin{
 		\draw[thick] (0,0) .. controls (0,0.5) and (-0.5,0.75) .. (-0.5,1);
 		\draw[thick] (-1,0) .. controls (-1,0.5) and (-0.5,0.75) .. (-0.5,1);
 		\fill (0,0) circle(2pt);
 		\fill (-1,0) circle(2pt);
 		\fill (-0.5,1) circle(2pt);
 		\draw[thick] (-2,0) -- (-2,1);
 		\fill (-2,0) circle(2pt);
 		\fill (-2,1) circle(2pt);
 		\fill (-2.1,0.4) -- (-1.9,0.4)--(-1.9,0.6) --(-2.1,0.6)--cycle;
 		\draw[thick] (-0.5,1) .. controls (-0.5,1.5) and (-1.25,1.75) .. (-1.25,2);
 		\draw[thick] (-2,1) .. controls (-2,1.5) and (-1.25,1.75) .. (-1.25,2);
 		\fill (-1.25,2) circle(2pt);
 		}\ .
\end{align}

Hom-associativity of type II ($\mu(\alpha(\mu(a,b)),c) = \mu(a,\alpha(\mu(b,c)))$ is given by the following diagrams:
\begin{align}\label{eq:homassIIdessinLangloisRemillard}
	\tikzdessin{
 		\draw[thick] (0,0) .. controls (0,0.5) and (0.5,0.75) .. (0.5,1);
 		\draw[thick] (1,0) .. controls (1,0.5) and (0.5,0.75) .. (0.5,1);
 		\fill (0,0) circle(2pt);
 		\fill (1,0) circle(2pt);
 		\fill (0.5,1) circle(2pt);
 		\draw[thick] (2,0) -- (2,1);
 		\fill (2,0) circle(2pt);
 		\draw[thick] (0.5,1) -- (0.5,2);
 		\fill (0.6,1.4) -- (0.4,1.4)--(0.4,1.6) --(0.6,1.6)--cycle;
 		\fill (0.5,2) circle(2pt);
 		\fill (2,2) circle(2pt);
 		\draw[thick] (2,1)--(2,2);
 		\draw[thick] (0.5,2) .. controls (0.5,2.5) and (1.25,2.75) .. (1.25,3);
 		\draw[thick] (2,2) .. controls (2,2.5) and (1.25,2.75) .. (1.25,3);
 		\fill (1.25,3) circle(2pt);
 		} \quad &=\quad 
 		\tikzdessin{
 		\draw[thick] (0,0) .. controls (-0,0.5) and (-0.5,0.75) .. (-0.5,1);
 		\draw[thick] (-1,0) .. controls (-1,0.5) and (-0.5,0.75) .. (-0.5,1);
 		\fill (0,0) circle(2pt);
 		\fill (-1,0) circle(2pt);
 		\fill (-0.5,1) circle(2pt);
 		\draw[thick] (-2,0) -- (-2,1);
 		\fill (-2,0) circle(2pt);
 		\draw[thick] (-0.5,1) -- (-0.5,2);
 		\fill (-0.6,1.4) -- (-0.4,1.4)--(-0.4,1.6) --(-0.6,1.6)--cycle;
 		\fill (-0.5,2) circle(2pt);
 		\fill (-2,2) circle(2pt);
 		\draw[thick] (-2,1)--(-2,2);
 		\draw[thick] (-0.5,2) .. controls (-0.5,2.5) and (-1.25,2.75) .. (-1.25,3);
 		\draw[thick] (-2,2) .. controls (-2,2.5) and (-1.25,2.75) .. (-1.25,3);
 		\fill (-1.25,3) circle(2pt);
 		}\ .
\end{align}

 If now the map $\alpha$ is taken to be some anti-involution $\iota$, one will need another application $\sigma : M\times M \to M\times M$ that switches two elements ($\sigma((a,b)) = (b,a)$): 
 \begin{align}
 	\sigma \longrightarrow \tikzdessin{ 
 	\draw[thick] (0,0) -- (1,1);
 	\draw[thick] (1,0) -- (0,1);
 	\fill (0,0) circle(2pt);\fill (0,1) circle(2pt);\fill (1,0) circle(2pt);\fill (1,1) circle(2pt);
 	}
 \end{align}
and add two rules ($\iota\circ \iota = \mathrm{id}$ and $\iota(\mu(a,b)) = \mu(\iota(b),\iota(a))$):
\begin{align}\label{eq:morphismisantiinvolutionLangloisRemillard}
	\tikzdessin{ 
	\draw[thick] (0,0) -- (0,1);
 	\fill (0,0) circle(2pt);
    \fill (0,1) circle(2pt);
 	\fill (0.1,0.4) -- (-0.1,0.4)--(-0.1,0.6) --(0.1,0.6)--cycle;
 	\draw[thick] (0,1) -- (0,2);
 	\fill (0,1) circle(2pt);
    \fill (0,2) circle(2pt);
 	\fill (0.1,1.4) -- (-0.1,1.4)--(-0.1,1.6) --(0.1,1.6)--cycle;
	}\ &= \ \tikzdessin{
	\draw[thick] (0,0) -- (0,1);
 	\fill (0,0) circle(2pt);
    \fill (0,1) circle(2pt);
 	}\ , & 
 	\tikzdessin{ 
 	\draw[thick] (0,0) .. controls (0,0.5) and (0.5,0.75) .. (0.5,1);
    \draw[thick] (1,0) .. controls (1,0.5) and (0.5,0.75) .. (0.5,1);
    \fill (0,0) circle(2pt);
    \fill (1,0) circle(2pt);
    \fill (0.5,1) circle(2pt);
    \draw[thick] (0.5,1) -- (0.5,2);
 	\fill (0.5,2) circle(2pt);
 	\fill (0.6,1.4) -- (0.4,1.4)--(0.4,1.6) --(0.6,1.6)--cycle;    
 	} \quad &= \quad 
 	\tikzdessin{ 
 	\draw[thick] (0,0) -- (1,1);
 	\draw[thick] (1,0) -- (0,1);
 	\fill (0,0) circle(2pt);\fill (0,1) circle(2pt);\fill (1,0) circle(2pt);\fill (1,1) circle(2pt);
 	\draw[thick] (0,1) -- (0,2);
 	\fill (0,1) circle(2pt);
    \fill (0,2) circle(2pt);
 	\fill (0.1,1.4) -- (-0.1,1.4)--(-0.1,1.6) --(0.1,1.6)--cycle;
 	\draw[thick] (1,1) -- (1,2);
 	\fill (1,1) circle(2pt);
    \fill (1,2) circle(2pt);
 	\fill (1.1,1.4) -- (0.9,1.4)--(0.9,1.6) --(1.1,1.6)--cycle;
 	\draw[thick] (0,2) .. controls (0,2.5) and (0.5,2.75) .. (0.5,3);
    \draw[thick] (1,2) .. controls (1,2.5) and (0.5,2.75) .. (0.5,3);
    \fill (0.5,3) circle(2pt);
 	}\ .
\end{align}

Multiple consequences can be derived from those last equations\footnote{Lars Hellstr\"om let a program he wrote for such investigation run for some hours and sent us back around one hundred lemmas. Unfortunately the amount of such lemmas is infinite and only when a previously known set of axiom is reached can a conclusion be drawn.}. For example, proving the hierarchy of Fr\'egier and Gohr amounts to diagrammatic consideration. As an example, the following prove that type I\textsubscript{1} and type II are equivalent only for unital Hom-associative algebras. 

Indeed, unitality amounts in adding one distinguished element $\circ$ such that the following condition is present for the multiplication $\mu$:
\begin{align}
	\tikzdessin{\draw[thick] (0,0.08) .. controls (0,0.5) and (0.5,0.75) .. (0.5,1);
    \draw[thick] (1,0) .. controls (1,0.5) and (0.5,0.75) .. (0.5,1);
    \draw (0,0) circle(2pt);
    \fill (1,0) circle(2pt);
    \fill (0.5,1) circle(2pt);} \ = \tikzdessin{\draw[thick] (0,0) .. controls (0,0.5) and (0.5,0.75) .. (0.5,1);
    \draw[thick] (1,0.08) .. controls (1,0.5) and (0.5,0.75) .. (0.5,1);
    \fill (0,0) circle(2pt);
    \draw (1,0) circle(2pt);
    \fill (0.5,1) circle(2pt);} \ = \  \tikzdessin{\draw[thick] (0,0) -- (0,1);
 		\fill (0,0) circle(2pt);
 		\fill (0,1) circle(2pt);} \ .
\end{align}
With this, one gets from equation~\eqref{eq:homassIdessinLangloisRemillard} (or from equation~\eqref{eq:homassIIdessinLangloisRemillard}) by placing the element $\circ$ in the second place
\begin{align}
	\tikzdessin{
 		\draw[thick] (0.5,0) -- (0.5,1);
 		\fill (0.5,0) circle(2pt);
 		\fill (0.5,1) circle(2pt);
 		\draw[thick] (2,0) -- (2,1);
 		\fill (2,0) circle(2pt);
 		\fill (2,1) circle(2pt);
 		\fill (2.1,0.4) -- (1.9,0.4)--(1.9,0.6) --(2.1,0.6)--cycle;
 		\draw[thick] (0.5,1) .. controls (0.5,1.5) and (1.25,1.75) .. (1.25,2);
 		\draw[thick] (2,1) .. controls (2,1.5) and (1.25,1.75) .. (1.25,2);
 		\fill (1.25,2) circle(2pt);
 		} \quad &=\quad 
 		\tikzdessin{
 		\draw[thick] (-0.5,0) -- (-0.5,1);
 		\fill (-0.5,0) circle(2pt);
 		\fill (-0.5,1) circle(2pt);
 		\draw[thick] (-2,0) -- (-2,1);
 		\fill (-2,0) circle(2pt);
 		\fill (-2,1) circle(2pt);
 		\fill (-2.1,0.4) -- (-1.9,0.4)--(-1.9,0.6) --(-2.1,0.6)--cycle;
 		\draw[thick] (-0.5,1) .. controls (-0.5,1.5) and (-1.25,1.75) .. (-1.25,2);
 		\draw[thick] (-2,1) .. controls (-2,1.5) and (-1.25,1.75) .. (-1.25,2);
 		\fill (-1.25,2) circle(2pt);
 		}\ .
\end{align}
It remains only to apply this new rule to equation~\eqref{eq:homassIIdessinLangloisRemillard} to obtain type I\textsubscript{1} Hom-associativity (or on equation~\eqref{eq:homassIdessinLangloisRemillard} to obtain type II Hom-associativity). 

It is possible to retrieve all their hierarchy in a similar fashion.

\begin{acknowledgement}
The author wish to thank Lars Hellstr\"{o}m for implementing a diagrammatic calculus for Hom-associative algebras of type II and sharing examples of the generated lemmas, and Asmus Bisbo for his comments on an earlier draft. This research benefited from a scholarship from the Fonds de recherche du Qu\'{e}bec -- Nature et technologies 270527 and from the EOS Research Project 30889451. This support is gratefully acknowledged.
\end{acknowledgement}

\bibliographystyle{plain}
\bibliography{Langlois-Remillard_Alexis_2}
\end{document}